\documentclass[twoside,11pt]{article}

\usepackage{amsmath}
\usepackage{amssymb}
\usepackage{amsthm}
\usepackage{indentfirst}
\usepackage[hidelinks]{hyperref}
\usepackage{paralist}
\usepackage{geometry}
\usepackage{setspace,lipsum}
\usepackage[title]{appendix}
\usepackage{color}
\usepackage{multimedia}
\usepackage{multirow}
\usepackage{bbm}
\usepackage{graphics}
\usepackage{graphicx}
\usepackage{booktabs}
\usepackage{romannum}
\usepackage{grffile}
\usepackage{float}
\usepackage{url}
\usepackage{mathabx}
\usepackage{supertabular}
\usepackage{rotating}
\usepackage{pdflscape}
\usepackage{verbatim}
\usepackage{listings}
\usepackage{xcolor}
\usepackage{multicol}
\usepackage[shortlabels]{enumitem} 
\usepackage{caption,subcaption}
\usepackage{algorithm}
\usepackage{algpseudocode}
\usepackage{tikz} 
\usepackage[round]{natbib}
\usepackage{authblk}
\usepackage{etoolbox}
\graphicspath{{figure/}}
\newcommand{\bm}[1]{\boldsymbol{#1}}

\DeclareMathOperator*{\argmax}{arg\,max}
\DeclareMathOperator*{\argmin}{arg\,min}
\newtheorem{assumption}{Assumption}

\newtheorem{theorem}{Theorem}[section]
\newtheorem{lemma}[theorem]{Lemma}
\newtheorem{remark}[theorem]{Remark}

\algblock{Input}{EndInput}
\algnotext{EndInput}
\algblock{Output}{EndOutput}
\algnotext{EndOutput}

\usetikzlibrary{fit,positioning}

\begin{document}

\pagenumbering{arabic}
\providecommand{\keywords}[1]
{
\small	
\textbf{\textit{Keywords:}} #1
}

\title{Continuous-Time Mean-Variance Portfolio Selection Reinforcement Learning Problems for a Markov Regime-Switching System with No-Shorting Constraints 
}

\author{Minghui Zhang \thanks{ School of Mathematics, Southeast University, Nanjing, 211189, P. R. China. E-mail: minghuibest@163.com}
~ ~ ~ 
Xun Li \thanks{ Department of Applied Mathematics, Hong Kong Polytechnic University, Hong Kong, P. R. China. E-mail: li.xun@polyu.edu.hk}
~ ~ ~ 
Xin Zhang \thanks{Corresponding author: School of Mathematics, Southeast University, Nanjing, 211189, P. R. China. E-mail: xzhangseu@seu.edu.cn}
}
\date{}
\maketitle

\let\thefootnote\relax

\begin{abstract}
\singlespacing
This paper is the first to study the continuous-time mean–variance portfolio selection problem with no-shorting constraints in a Markov regime-switching market using reinforcement learning (RL). The main challenge stems from the interaction between random regime shifts and portfolio non-negativity constraints, which typically causes the coupled Hamilton–Jacobi–Bellman (HJB) equations to lose smoothness across different state regions. Additionally, directly imposing no-shorting constraints on randomized policies disrupts their analytically tractable Gaussian structure, complicating both exploratory modeling and reinforcement learning. To address these challenges, we propose an auxiliary exploratory approach that limits the policy to a Gaussian distribution and applies the no-shorting constraint solely to its mean. We explicitly characterize the optimal policy for this auxiliary problem and demonstrate that, by appropriately selecting the exploration variance, its optimal mean policy exactly matches the optimal policy of the original constrained mean–variance problem. Building upon this analytical framework, we develop a model-free reinforcement learning algorithm that learns the optimal portfolio policy directly from trajectory data. Numerical experiments support the theoretical results and show that the proposed method can effectively recover the optimal no-shorting portfolio in a regime-switching market.
\end{abstract}

\small{\textit{\textbf{Keywords:}}Continuous-Time RL, Mean-Variance portfolio optimization, No-Shorting, Markov regime-switching, Model-Free Policy Gradient, RL Algorithm}

\section{Introduction.}\label{section_introduction}
Mean–variance portfolio selection is one of the fundamental problems in mathematical finance. It was originally introduced by \citet{markowitz1952portfolio} and later extended to the classical continuous-time framework by \citet{merton1972analytic}. Its practical relevance and rich analytical structure have motivated extensive research on this problem; see, for example, \citet{yong1999stochastic, zhou2000continuous}, and the references therein. In the finance field, investors could face two market regimes, one of which represents a bull market with price increases, while the other regime represents a bear market with price drops. Therefore, the regime-switching type portfolio selection problem is of great interest and importance in financial investment. Meanwhile, portfolio constraints are indispensable in realistic investment problems, among which no-shorting is one of the most basic and widely used constraints. Such constraints significantly increase the complexity of the classical mean-variance problem and naturally lead to constrained stochastic control problems. 
To effectively capture macroeconomic cycles and abrupt transitions in market states, this paper considers the continuous-time mean-variance portfolio selection problem with no-shorting constraints for a Markov regime-switching system from a model-free reinforcement learning perspective. 

     Regime-switching models have been widely used in portfolio selection.  \citet{zhou2003markowitz} introduced the Markov regime-switching framework to the continuous-time mean–variance portfolio selection problem, providing a theoretical foundation for this line of research. Subsequently, later studies extended regime-switching models to more general investment decision problems; see \citet{Chen2008markowitz,Sotomayor2009explicit,Canakoglu2010portfolio}.
     Compared with models with deterministic market parameters, regime-switching substantially increases the analytical complexity of the portfolio selection problem. Because different market regimes correspond to different value functions, dynamic programming typically yields a system of HJB equations. In particular, when the risk-free interest rate is also regime-dependent, this coupling structure is further strengthened, making it difficult to obtain an analytical solution to the corresponding equations.
     To avoid this difficulty, many studies on regime-switching portfolio problems assume a deterministic risk-free interest rate; see, for example, \citet{wei2017time,wu2015nash,zhang2016mean}. To relax this assumption, \citet{chen2020continuous} let the parameters of both the risk-free and risky assets vary across regimes and used stochastic Riccati equations to address the resulting coupling, extending the regime-switching framework to a more general market setting.

     Recent research in stochastic control (\cite{dai2023learning, denkert2024control, jia2021policy, jia2021policypg, jia2023q, wang2020continuous, wang2020reinforcement, wu2024reinforcement}) has demonstrated the advantages of applying RL methods to classical stochastic control problems. In particular, \citet{wang2020reinforcement} developed an entropy-regularized exploratory control framework for continuous-time stochastic control, characterized the corresponding optimal exploratory policy, and established the asymptotic equivalence between the exploratory and classical control solutions. Based on this framework, \citet{wang2020continuous} introduced the exploratory mean-variance (EMV) portfolio selection problem and developed a corresponding policy improvement method. Subsequently, Jia and  Zhou \cite{jia2021policy, jia2021policypg, jia2023q} systematically investigated three fundamental modules of reinforcement learning—policy evaluation, policy gradient, and continuous-time q-learning—aiming to establish a rigorous theoretical foundation for model-free continuous-time RL via the martingale approach. 
     On the theoretical side, \citet{tang2022exploratory} established well-posedness, regularity, and vanishing-exploration convergence for exploratory HJB equations. Regarding algorithmic convergence, \citet{huang2025sublinear} investigated the convergence of related algorithms for a class of continuous-time linear-quadratic (LQ) control problems, while \citet{huang2025convergence} proved the convergence of policy iteration for general entropy-regularized stochastic control problems over infinite horizons. Furthermore, \citet{sethi2025entropy} studied algorithms and asymptotic properties for entropy-regularized control, analyzed continuous-time policy mirror descent and entropy annealing, and quantified convergence for stochastic exit problems. In addition, \citet{guo2026fast} proposed regularized policy gradient and iterative policy optimization methods for entropy-regularized infinite-horizon LQ control and established linear and partial superlinear convergence rates.
     Additionally, \citet{bourdais2026entropy} extended the entropy penalty approach to stochastic control problems formulated on the space of probability measures. 

     At the same time, balancing control constraints with sufficient exploration has become an important issue in continuous-time reinforcement learning. \citet{chau2026continuous} studied the continuous-time expected utility maximization problem with portfolio constraints and showed that the optimal exploration strategy transforms from a Gaussian distribution to a truncated Gaussian distribution. More closely related to our approach, \citet{dai2025data} constructed an auxiliary control problem via Gaussian policy randomization and and prove that the mean of its optimal Gaussian policy solves the original Merton problem, but this problem is unconstrained. \citet{li2026dynamic} extended this idea to the constrained setting and proved that a suitably designed auxiliary exploratory problem can preserve the analytical tractability of Gaussian policies while ensuring that the mean of the optimal policy coincides with the optimal control of the original problem.

     Recent studies have extended continuous-time reinforcement learning to Markov regime-switching environments and developed corresponding exploratory control and learning frameworks. \citet{wu2024reinforcement} formulated the regime-switching mean–variance portfolio problem as a partially observable Markov decision process and used Wonham filtering to infer the unobservable market regime, but their approach still relies on market parameters that are known or estimated in advance. Subsequently, \citet{chen2025exploratory} developed an exploratory mean–variance framework for regime-switching, based on the martingale property of the optimal value function, and used martingale orthogonality to construct a learning approach that does not explicitly require knowledge of the true market parameters. \citet{zhang2026continuous} further extended Shannon entropy to Tsallis entropy and developed corresponding continuous-time q-learning methods, while \citet{xing2026exploratory} incorporated a constant elasticity of variance model into the regime-switching reinforcement learning framework to capture more general state-dependent volatility dynamics.

     Although continuous-time reinforcement learning has made significant progress in mean–variance portfolio optimization, it still faces important theoretical challenges in extending existing methods to markets with both Markov regime-switching and no-shorting constraints. In the Markov regime-switching setting, the mean–variance problem gives rise to a system of coupled HJB equations, while the no-shorting constraint further restricts the portfolio control to the nonnegative closed convex cone ($u\geq 0$). The combination of regime switching and no-shorting constraints leads to a piecewise structure in the optimal control that depends on the state, while the corresponding value function generally no longer admits the globally smooth quadratic form characteristic of unconstrained LQ problems. This not only makes the coupled HJB system harder to solve analytically, but also makes it difficult to apply policy evaluation and Actor–Critic algorithms that rely on the value function's smoothness and explicit structure.

     Another key challenge lies in constructing the exploratory policy. In unconstrained entropy-regularized quadratic control problems, the optimal exploratory policy typically admits an analytically Gaussian form. Under the no-shorting constraint, however, a standard Gaussian policy necessarily assigns positive probability to negative controls and therefore cannot satisfy the strict requirement ($u\geq 0$). If we further restrict the support of the random strategy to the non-negative region, for example, a truncated Gaussian distribution, we can guarantee control feasibility, but this destroys the analytical convenience of Gaussian randomization and the quadratic control structure. As a result, both policy optimization and the associated HJB equations may no longer admit closed-form representations. The central challenge is therefore to preserve the analytical tractability of Gaussian exploration while respecting the nonnegative control constraint and establishing a rigorous connection between the exploratory problem and the original constrained control problem.

     The main contributions of this paper are summarized as follows:
     \begin{itemize}
        \item To the best of our knowledge, this is the first study of the continuous-time mean–variance portfolio selection problem with no-shorting constraints in a Markov regime-switching market from a reinforcement learning perspective. The main difficulty is that the nonnegative portfolio constraint destroys the Gaussian structure that makes exploratory control analytically tractable. To address this difficulty, we propose an auxiliary Gaussian exploratory formulation where the policy remains Gaussian, but only its mean is constrained to lie in the nonnegative orthant $\mathbb{R}_{+}^{m}$. This construction preserves the tractability of Gaussian exploration while ensuring that the ultimately implemented mean policy satisfies the no-shorting constraint.

        \item We establish an exact recovery result that connects the auxiliary exploratory problem to the original constrained mean-variance problem. Specifically, we show that by choosing the exploration variance appropriately, the mean of the optimal Gaussian policy in the auxiliary problem exactly matches the optimal control for the original constrained problem. Therefore, the proposed Gaussian exploration is not just a computational approximation; it is equivalent to the original constrained problem with respect to the mean of the optimal Gaussian policy. This equivalence provides strong theoretical justification for using Gaussian exploration to address strictly constrained portfolio problems.

        \item Building on the above theoretical results, together with the martingale-based policy evaluation method of \citet{jia2021policy} and the policy gradient theory of \citet{jia2021policypg}, we develop a model-free Actor–Critic framework for the constrained regime-switching problem. The proposed method requires no explicit knowledge of the market information and avoids repeatedly solving the coupled HJB equations during learning. Instead, it learns the optimal constrained portfolio directly from observed trajectories. Numerical experiments validate the theoretical results and demonstrate that the proposed method can effectively recover the optimal no-shorting portfolio in a Markov regime-switching market.
    \end{itemize}

     The remainder of this paper is organized as follows: Section \ref{sec2} formulates the continuous-time mean-variance problem with no-shorting constraints for a Markov regime-switching system and introduces the auxiliary exploratory problem. Section \ref{sec3} provides an explicit solution to the auxiliary problem and compares it with the target problem. Section \ref{sec4} derives and designs a model-free, continuous-time RL algorithm within the restricted policy space. Section \ref{sec5} validates the effectiveness of the algorithm through numerical experiments. Section \ref{sec6} concludes the paper.

\section{Problem Formulation and Preliminaries.}\label{sec2}

\subsection{Mean-Variance Problem with No-Shorting Constraints for a Markov regime-switching system.}
For a fixed time horizon $T>0$,  let $(\Omega, \mathcal{F}, \mathbb{F}, \mathbb{P})$ be a filtered probability space with the filtration $\mathbb{F}:=\{\mathcal{F}_t, 0 \leq t \leq T\}$ on which are defined a standard $m$-dimensional Brownian motion $W=\{W_t^1,\cdots, W_t^m, 0 \leq t \leq T\}$ and a continuous-time irreducible Markov chain $\alpha=\{\alpha_t, 0 \leq t \leq T\}$ in a finite state space $\mathcal{M}=\{1, 2, \ldots, L\}$ with $L>1$. We assume $W_t$ and $\alpha_t$ are independent processes.
The generator matrix (transition rate matrix) of the Markov chain is denoted by $Q = (q_{ij})_{L \times L}$, where for $i \neq j$, $q_{ij} > 0$ represents the transition rate from state $i$ to state $j$, and $q_{ii} = -\sum_{j \neq i} q_{ij}$. 
Define the filtrations $\mathcal{F}_t=\sigma\{W_s, \alpha_s, 0 \leq s \leq t\} \bigvee \mathcal{N}$ and $\mathcal{F}_t^W=\sigma\{W_s, 0 \leq s \leq t\} \bigvee \mathcal{N}$, where $\mathcal{N}$ is the totality of all the $\mathbb{P}$-null sets of $\mathcal{F}$.

{\it Notation}. We use the following notation throughout the paper:
\begin{itemize}
\item $\mathcal{P}(\mathbb{R}^m)$ : the set of density functions on $\mathbb{R}^m$;
\item $\mathcal{P}(\mathbb{R}_{+}^m)$ : the set of density functions on $\mathbb{R}_{+}^m$;
\item $\mathcal{P}_{+}(\mathbb{R}^m)$ : the subset of $\mathcal{P}(\mathbb{R}^m)$ consisting of distributions with nonnegative mean;
\item $\mathbb{S}_{++}^m$ is the space of $m \times m$ symmetric positive definite matrices;
\item $L_{\mathcal{F}}^2(0, T ; \mathbb{R})=\bigg\{\varphi:[0, T] \times \Omega \rightarrow \mathbb{R} \mid \varphi(\cdot)$ is an $\{\mathcal{F}_t\}_{t \geq 0}$-adapted process with the norm $\|\varphi\|=\big(\mathbb{E} \int_0^T|\varphi(t)|^2 d t\big)^{\frac{1}{2}}<\infty\bigg\}$;
\item $L_{\mathcal{F}}^{\infty}(0, T ; \mathbb{R})=\bigg\{\varphi:[0, T] \times \Omega \rightarrow \mathbb{R} \mid \varphi(\cdot)$ is an $\left\{\mathcal{F}_t\right\}_{t \geq 0}$-adapted essentially bounded process\bigg\}.
\end{itemize}
    
Consider a financial market consisting of one risk-free asset that satisfies
$$
d S_t^0=S_t^0 r\left(t, \alpha_t\right) d t, \quad t \in[0, T], S_0^0=s_0>0,
$$
and $m$ risky assets whose prices are governed by geometric Brownian motions 
$$
d S_t^k=S_t^k b^k\left(t, \alpha_t\right) d t+S_t^k \sum_{l=1}^m \sigma^{k l}\left(t, \alpha_t\right) d W_t^l, \quad t \in[0, T], S_0^k=s_k>0, \quad k \in\{1,2, \ldots, m\}.
$$ 
where, for every $k=1, \ldots, m$ and $i \in \mathcal{M}$, $r(t, i)$ is the interest rate process, $b^k(t, i)$ and $\sigma^k(t, i):=\left(\sigma^{k 1}(t, i), \ldots, \sigma^{k n}(t, i)\right)^\top$ are the appreciation rate process and volatility rate process of the $k$-th risky security corresponding to a market regime $\alpha_t=i$.

Define the appreciation vector $b(t, i)=(b_1(t, i), \ldots, b_m(t, i))^\top$, and volatility matrix
$$\sigma(t, i)=(\sigma^1(t,i),\cdots,\sigma^m(t,i))^\top\equiv(\sigma^{k l}(t, i))_{m \times m}, \text{ for each }  i \in \mathcal{M}.
$$
In the rest of this paper, we assume that $r(\cdot, i), \mu^k(\cdot, i)$, and $\sigma^{k l}(\cdot, i)$ are deterministic, measurable, and bounded functions of $t$ on $[0, T]$, for all $k=1, \ldots, m, l=1, \ldots, m$, and $i \in \mathcal{M}$. In addition, there exists a constant $\delta>0$ such that $\sigma(t, i) \sigma(t, i)^\top \geq \delta I_m$ for all $t \in[0, T]$ and $i \in \mathcal{M}$, where $I_m$ is the $m \times m$ identity matrix.

Let $u(\cdot)=\{u_t, 0 \leq t \leq T\}$ be the portfolio (process), where $u_t:=\left(u_t^1, u_t^2, \ldots, u_t^m\right)^{\top}$ with $u_t^k$ being the {\it discounted value} invested in the $k$-th risky asset. Then, the dynamics of the {\it discounted wealth process} $\left\{x_t, 0 \leq t \leq T\right\}$ satisfies
\begin{equation}\label{eq:wealth-process}
d x_t=u_t^{\top} \sigma\left(t, \alpha_t\right)\left(\rho\left(t, \alpha_t\right) d t+d W_t\right), \quad t \in[0, T],\,\, x_0 \in \mathbb{R},\,\,\alpha_0 \in \mathcal{M} ,
\end{equation}
where $\rho(t, \alpha_t):=\sigma(t, \alpha_t)^{-1}\left(b(t, \alpha_t)-r(t, \alpha_t) \mathbf{1}\right)$ is the relative risk coefficient (market risk premium).

We study a continuous-time mean-variance portfolio selection problem under {\it no-shorting constraints} $u_t^k \geq 0$ for all $k \in \{1, \dots, m\}$ which is formulated as the following constrained optimization problem
\begin{equation}\label{eq:constrained-optimization-problem}
\min_{u(\cdot) \in L_{\mathcal{F}}^2(0, T ; \mathbb{R}_{+}^m)} \operatorname{Var}\left[x_T \mid x_0, \alpha_0\right], \quad \text { subject to (s.t.) } \quad \mathbb{E}\left[x_T \mid x_0, \alpha_0\right]=z,
\end{equation}
where the wealth–portfolio pair $(x(\cdot), u(\cdot))$ satisfies the dynamics \eqref{eq:wealth-process} and $z \in \mathbb{R}$ is a pre-specified {\it discounted expected return}. 
    
To ensure the problem is economically meaningful and practically well-posed, we impose the following assumption.
\begin{assumption}\label{Assump:appreciation-rate}
For all $t\in [0, T]$ and $i \in \mathcal{M}$, the appreciation rate strictly dominates the risk-free rate, i.e., $b(t, i) - r(t, i)\mathbf{1} > \mathbf{0}$ (component-wise). Furthermore, the initial wealth $x_0>0$, and the target expected return strictly exceeds the initial wealth, i.e., $z>x_0$.
\end{assumption}
    
To deal with the expectation constraint $\mathbb{E}[x_T\mid x_0, \alpha_0]=z$, we introduce a Lagrange multiplier $w \in \mathbb{R}$. Following \citet{zhou2000continuous}, the problem \eqref{eq:constrained-optimization-problem} is thus transformed into a standard stochastic optimal control problem parameterized by $w$:
\begin{equation}\label{eq:optimal-control-problem}
\min_{u(\cdot) \in L_{\mathcal{F}}^2(0, T ; \mathbb{R}_{+}^m)} J(u(\cdot) ; w) := \min_{u(\cdot) \in L_{\mathcal{F}}^2(0, T ; \mathbb{R}_{+}^m)} \mathbb{E}\left[\left(x_T-w\right)^2 \mid x_0, \alpha_0\right]-(w-z)^2.
\end{equation}
For each fixed $w \in \mathbb{R}$, let $u^*(\cdot; w)$ denote the optimal control of the unconstrained auxiliary problem \eqref{eq:optimal-control-problem}, and let $V(w) := J(u^*(\cdot; w); w)$ denote the corresponding optimal value function. Based on the stochastic linear-quadratic control framework and duality theory, \eqref{eq:constrained-optimization-problem} can be solved by maximizing the auxiliary value function over the multiplier $w$. Specifically, the optimal multiplier is given by
\begin{equation*}\label{eq:correct-multiplier}
w^* := \argmax_{w\in\mathbb{R}} V(w) = \argmax_{w\in\mathbb{R}} J(u^*(\cdot);w); w).
\end{equation*}
In particular, when evaluated at the optimal multiplier $w=w^*$, \eqref{eq:optimal-control-problem} and \eqref{eq:constrained-optimization-problem} share the exact same optimal control policy and optimal value. We refer to \citet{li2002Dynamic} for the rigorous proof of strong duality under no-shorting constraints.

\subsection{The Auxiliary Problem with Nonnegative Constraints.}
In practice, the decision maker (agent) often has limited or no knowledge of the environment (i.e., $b, \sigma, \rho, r$ are unknown). Reinforcement learning offers an efficient method for learning an optimal portfolio (control) for \eqref{eq:constrained-optimization-problem} in such unknown environments by repeatedly taking actions and interacting with the environment (trial-and-error procedure).
In continuous-time reinforcement learning, entropy regularization is commonly introduced to encourage policy exploration. \citet{wang2020continuous} applied this approach to the unconstrained mean–variance problem and showed that the optimal exploratory policy follows a Gaussian distribution. However, directly imposing the no-shorting constraint on the support of the randomized policy typically leads to a truncated Gaussian distribution, which complicates the analytical characterization of the value function and the associated coupled HJB equations. Following \citet{dai2025data}, we formulate an auxiliary exploratory problem by restricting the randomized policy to the Gaussian family and imposing the nonnegative constraint only on its mean.

Next, we establish the auxiliary problem corresponding to \eqref{eq:constrained-optimization-problem}. 
In the auxiliary problem, the portfolio (control) $u(\cdot)$ is replaced by a randomized portfolio process $\pi(\cdot)$, which is a distributional (density-function)-valued, progressively measurable stochastic process $\{\pi_t, 0 \leq t \leq T\}$. We introduce the nonnegative integral constraints 
$$
\int u \pi_t(u) d u \geq 0, \quad \forall t \in[0, T].
$$
In other words, at each time instant $t$, the distribution $\pi_t$ must have a nonnegative mean. For convenience, denote $\mathcal{P}_{+}(\mathbb{R}^m):=\left\{\pi \in \mathcal{P}(\mathbb{R}^m) \mid \int u \pi(u) d u \geq 0\right\}$. Then the integral constraints can also be written as $\pi_t\in \mathcal{P}_{+}\left(\mathbb{R}^m\right), \forall t\in [0, T]$. It is obvious that $\mathcal{P}\left(\mathbb{R}_{+}^m\right) \subset \mathcal{P}_{+}\left(\mathbb{R}^m\right) \subset \mathcal{P}\left(\mathbb{R}^m\right)$.

Following \citet{wang2020continuous}, for a randomized policy $\pi(\cdot)$, the corresponding exploratory dynamics is 
\begin{equation}\label{eq:exploratory-dynamics}
d x_t^\pi=\tilde{b}(t, \pi_t, \alpha_t) d t+\sum_{j=1}^m \tilde{\sigma}^j(t, \pi_t, \alpha_t) d W_t^j, \quad t \in[0, T],
\end{equation}
where $\tilde{b}(\cdot,\cdot,\cdot)$ and $\tilde{\sigma}(\cdot,\cdot,\cdot)$ are defined on $[0, T] \times \mathcal{P}(\mathbb{R}^m) \times \mathcal{M}$ are given by
$$
\tilde{b}(t, \pi, i):=\int_{\mathbb{R}^m}\left(\sigma(t, i) \rho(t, i)\right)^{\top} u \pi(u) du, \quad \tilde{\sigma}^l(t, \pi, i):=\left(\int_{\mathbb{R}^m}\left(\sum_{l=1}^m \sigma^{kl}(t, i) u_k\right)^2 \pi(u) du \right)^{1 / 2}.
$$

To formulate the auxiliary problem rigorously, we first specify the set of randomized feedback portfolios (policies). Consider a feedback policy $\pi:[0, T] \times \mathbb{R} \times \mathcal{M} \rightarrow \mathcal{P}_{+}(\mathbb{R}^m)$. Let $\{x_t^\pi, 0 \leq t \leq T\}$ denote the solution to \eqref{eq:exploratory-dynamics} with $\pi_t$ replaced by $\pi(t, x_t^\pi, \alpha_t)$. The policy $\pi$ is called admissible if its induced control process $\{\pi(t, x_t^\pi, \alpha_t), 0 \leq t \leq T\}$ is an admissible process that explicitly satisfies the nonnegative mean constraint. 
Fixing a measurable covariance function $\Sigma:[0, T] \times \mathbb{R}\times \mathcal{M} \rightarrow \mathbb{S}_{++}^m$, we denote by $\Pi$ the set of admissible Gaussian policies of the form $\pi(t, x, i)=\mathcal{N}(a(t, x, i), \Sigma(t, x, i))$, where $a:[0, T] \times \mathbb{R} \times \mathcal{M} \rightarrow \mathbb{R}_{+}^m$ is measurable. The corresponding auxiliary problem is 
\begin{equation}\label{eq:auxiliary-problem}
\min _{\pi(\cdot) \in \Pi} J(\pi(\cdot) ; w):=\mathbb{E}\left[\left(x_T^\pi-w\right)^2 \mid x_0, \alpha_0\right]-(w-z)^2.
\end{equation}
 
\begin{remark}
We clarify the relationship between the target problem \eqref{eq:optimal-control-problem} and the exploratory auxiliary problem \eqref{eq:auxiliary-problem} as follows:
\begin{itemize}
\item The Gaussian randomization in the auxiliary problem can be expressed as
$\pi(t, x, i) =\mathcal{N}(a(t, x, i), \Sigma(t, x, i))$, where $a(t, x, i) \in \mathbb{R}_{+}^m$. This can be interpreted as adding zero-mean Gaussian noise with covariance $\Sigma(t, x, i)$ to an admissible feedback control $a(t, x, i)$. Unlike randomization restricted to the constrained action space, this construction allows Gaussian exploration over the entire space $\mathbb{R}^m$, while the mean of the optimal policy remains feasible for the target problem \eqref{eq:optimal-control-problem}. 
\item Since $\Sigma(t, x, i)$ is positive definite, the Gaussian policy has full support on $\mathbb{R}^m$, and individual controls sampled during exploration may violate the no-shorting constraint. Nevertheless, its mean $a_t:= \int_{\mathbb{R}^m} u \, \pi_t(u) du \geq \mathbf{0}$ (component-wise) always satisfies the nonnegative constraint. Thus, the Gaussian noise serves only for exploration, while the mean of the optimal policy implements the investment control in the target problem \eqref{eq:optimal-control-problem}.
\item Under suitable choices of the variance function $\Sigma(\cdot, \cdot)$, the optimal mean process $a_t^*:= \int_{\mathbb{R}^m} u \, \pi_t^*(u) du$ of the auxiliary problem coincides exactly with the optimal control of the target problem \eqref{eq:optimal-control-problem}. This establishes a theoretical bridge between the exploratory framework and the original constraint problem, which is the main result of this paper and will be proved in the subsequent part.
\end{itemize}
\end{remark}

To verify our results, we establish the following lemma.
\begin{lemma}\label{lem:scaling-lemma}
Let $\Sigma \in \mathbb{S}_{++}^m$ be a strictly positive definite matrix, $p \in \mathbb{R}^m$ be a constant vector, and $s>0$ be a scaling factor. Then, for any given market regime $k \in \mathcal{M}$, the following properties hold:
$$
\min_{a \in \mathbb{R}_{+}^m} \left\{ \frac{1}{2} a^{\top} \Sigma a + sa^{\top}p \right\} = s^2 \cdot \min_{\mu \in \mathbb{R}_{+}^m} \left\{ \frac{1}{2} \mu^{\top} \Sigma \mu + \mu^{\top} p \right\}.
$$
Moreover, there is 
$$\begin{aligned}
\argmin_{a \in \mathbb{R}_{+}^m} \left\{ \frac{1}{2} a^{\top} \Sigma a + sa^{\top}p \right\} = s \cdot \argmin_{\mu \in \mathbb{R}_{+}^m} \left\{ \frac{1}{2} \mu^{\top} \Sigma \mu + \mu^{\top} p \right\},\\
\argmin_{a \in \mathbb{R}_{+}^m} \left\{ \frac{1}{2} a^{\top} \Sigma a - sa^{\top}p \right\} = s \cdot \argmin_{\mu \in \mathbb{R}_{+}^m} \left\{ \frac{1}{2} \mu^{\top} \Sigma \mu - \mu^{\top} p \right\}.
\end{aligned}$$
\end{lemma}
\begin{proof}{proof}
The proof has two parts: verifying the relationship between the minimizer and deriving the scaled minimum value.
        
Let the objective function of the problem on the left-hand side be denoted by $f_s(a):= \frac{1}{2} a^{\top} \Sigma a + sa^{\top}p$. The minimizer is defined as:
$$
a^*(s) = \argmin_{a \in \mathbb{R}_+^m} \left( \frac{1}{2} a^{\top} \Sigma a + sa^{\top}p \right).
$$
Since $s > 0$, we can perform a change of variables by letting $a = s \cdot \mu$. Since $\mathbb{R}_+^m$ is a cone and $s > 0$, the mapping $a = s\mu$ is a bijection on $\mathbb{R}_+^m$ (i.e., $\mu \in \mathbb{R}_+^m \iff s\mu \in \mathbb{R}_+^m$ for $s > 0$). Substituting $a = s\mu$ into the objective function, we obtain
$$\begin{aligned}
f_s(s\mu) = &\frac{1}{2} (s\mu)^{\top} \Sigma (s\mu) + s (s\mu)^{\top} p= s^2 \left( \frac{1}{2} \mu^{\top} \Sigma \mu + \mu^{\top} p \right).
\end{aligned}$$
Let $f_1(\mu) := \frac{1}{2} \mu^{\top} \Sigma \mu + \mu^{\top} p$ be the objective function for the case $s=1$. Thus, $f_s(s\mu) = s^2 f_1(\mu)$. It follows that
$$
\min_{a \in \mathbb{R}_+^m} f_s(a) = \min_{\mu \in \mathbb{R}_+^m} s^2 f_1(\mu)= s^2\cdot\min_{\mu \in \mathbb{R}_+^m}  f_1(\mu),
$$
i.e.,
$
\min_{a \in \mathbb{R}_{+}^m} \left\{ \frac{1}{2} a^{\top} \Sigma a + sa^{\top}p \right\} = s^2 \cdot \min_{\mu \in \mathbb{R}_{+}^m} \left\{ \frac{1}{2} \mu^{\top} \Sigma \mu + \mu^{\top} p \right\}$.       
Set $a^*\in\argmin\limits_{a \in \mathbb{R}_+^m} f_s(a)$, and define $\mu^* =\frac{a^*}{s}$. Next, we only need to prove $\mu^* \in \argmin\limits_{\mu \in \mathbb{R}_+^m} f_1(\mu)$.
Note that
$$
f_s(s\cdot \mu^*)=f_s(a^*)=\min_{a \in \mathbb{R}_+^m} f_s(a)= \min_{\mu \in \mathbb{R}_+^m} f_s(s\mu).
$$
Using the identity $f_s(s\mu) = s^2 f_1(\mu)$, we obtain
$$    
s^2 f_1(\mu^*) = f_s(s \cdot \mu^*)=\min_{\mu \in \mathbb{R}_+^m} f_s(s\mu)= \min_{\mu \in \mathbb{R}_+^m} \left( s^2 f_1(\mu) \right) = s^2 \min_{\mu \in \mathbb{R}_+^m} f_1(\mu).
$$
Since $s > 0$, dividing both sides by $s^2$ yields
$$ 
f_1(\mu^*) = \min_{\mu \in \mathbb{R}_+^m} f_1(\mu).
$$        
Therefore, we get $\mu^* \in \argmin\limits_{\mu \in \mathbb{R}_+^m} f_1(\mu)$. By the bijective nature of the mapping $a=s\mu$, this establishes the equivalence and completes the proof of
$$    
\argmin\limits_{a \in \mathbb{R}_+^m} f_s(a) = s\cdot \argmin\limits_{\mu \in \mathbb{R}_+^m} f_1(\mu),
$$
i.e., $\argmin_{a \in \mathbb{R}_{+}^m} \left\{ \frac{1}{2} a^{\top} \Sigma a + sa^{\top}p \right\} = s \cdot \argmin_{\mu \in \mathbb{R}_{+}^m} \left\{ \frac{1}{2} \mu^{\top} \Sigma \mu + \mu^{\top} p \right\}$.

Similarly, $\argmin\limits_{a \in \mathbb{R}_{+}^m} \left\{ \frac{1}{2} a^{\top} \Sigma a - sa^{\top}p \right\} = s \cdot \argmin\limits_{\mu \in \mathbb{R}_{+}^m} \left\{ \frac{1}{2} \mu^{\top} \Sigma \mu - \mu^{\top} p \right\}$ can be proved.
\end{proof}

\begin{remark}
Under the Markov regime-switching model, we substitute $\Sigma(t, x, k)$ with the regime-dependent covariance matrix $\sigma(t, k)\sigma(t, k)^{\top}$ and $p$ with the corresponding risk premium. Since the scaling factor $s$ (typically representing the wealth deviation $|x-w|$) is a scalar at any given time $t$, Lemma \ref{lem:scaling-lemma} ensures that the value function maintains its piecewise quadratic structure across all regimes $k \in \mathcal{M}$. This reduces the coupled PDE system to a system of tractable ODEs.
\end{remark}

\section{Explicit Solutions of the Auxiliary Problem.}\label{sec3}
    
In this section, we derive an explicit solution to the auxiliary problem \eqref{eq:auxiliary-problem} and connect it to the target problem \eqref{eq:optimal-control-problem} by choosing the variance function $\Sigma(\cdot, \cdot, \cdot)$ appropriately. This characterization provides a theoretical basis for the model-free algorithm developed in the next section.

Define the value function $\tilde{V}: [0, T] \times \mathbb{R} \times \mathcal{M} \to \mathbb{R}$ of \eqref{eq:auxiliary-problem} by
$$
\tilde{V}(t, x, i ; w):=\min_{\pi \in \Pi} \mathbb{E}\left[\left(x_T^\pi-w\right)^2-(w-z)^2 \mid x_t^\pi=x, \alpha_t=i\right].
$$
For a given regime $i \in \mathcal{M}$, let $v(t, x, i)$ denote the candidate value function. The corresponding Hamiltonian is
$$\begin{aligned} 
\mathcal{H}(t, x, i, \pi) &:= \frac{\partial}{\partial x} v(t, x, i ; w) \tilde{b}(t, \pi, i)+\frac{1}{2} \frac{\partial^2}{\partial x^2} v(t, x, i ; w)\left\|\tilde{\sigma}(t, \pi, i)\right\|^2 \\
& =\int_{\mathbb{R}^m}\left(\frac{\partial }{\partial x} v(t, x, i ; w) u^{\top} \sigma(t, i) \rho(t, i)+\frac{1}{2} \frac{\partial^2}{\partial x^2} v(t, x, i ; w) u^{\top} \sigma(t, i) \sigma(t, i)^{\top} u\right) \pi(u) du,
\end{aligned}$$
where $(t, x, i, \pi) \in[0, T] \times \mathbb{R} \times \mathcal{M} \times \mathcal{P}_{+}(\mathbb{R}^m)$. Consequently, the HJB equation for $v(t, x, i)$ is 
\begin{equation}\label{eq:HJB}
\left\{\begin{array}{l}
\min\limits_{\pi \in \mathcal{P}_{+}(\mathbb{R}^m)} \int_{\mathbb{R}_{+}^m}\bigg\{\frac{\partial}{\partial x} v(t, x, i) u^{\top} \sigma(t, i) \rho(t, i)+\frac{1}{2} \frac{\partial^2}{\partial x^2} v(t, x, i) u^{\top} \sigma(t, i) \sigma(t, i)^{\top} u\bigg\} \pi(u) d u\\
\quad+\frac{\partial}{\partial t} v(t, x, i)+\sum\limits_{j=1}^L q_{ij} v(t, x, j)=0, \\
v(T, x, i)=(x-w)^2-(w-z)^2, \quad \forall i \in \mathcal{M}.
\end{array}\right.
\end{equation}
For each $(t, x, i)$, the randomized policy is restricted to $\pi=\mathcal{N}(a, \Sigma(t, x, i))$. Substituting it into \eqref{eq:HJB}, we obtain the following HJB equation:
\begin{equation}\label{eq:optimal-HJB-equation}
\left\{\begin{array}{l}
\frac{\partial}{\partial t} v(t, x, i)+\min\limits_{a \in \mathbb{R}_{+}^m}\left(\frac{\partial}{\partial x} v(t, x, i) a^{\top} \sigma(t, i) \rho(t, i)+\frac{1}{2} \frac{\partial^2}{\partial x^2} v(t, x, i) a^{\top} \sigma(t, i) \sigma(t, i)^{\top} a\right)\\ \quad\quad\quad\quad\quad\quad\quad\quad+\sum\limits_{j=1}^L q_{ij} v(t, x, j)+\frac{1}{2} \frac{\partial^2}{\partial x^2} v(t, x, i) \operatorname{tr}\{\Sigma(t, x, i) \sigma(t, i) \sigma(t, i)^{\top}\}=0, \\
v(T, x, i)=(x-w)^2-(w-z)^2, \quad \forall i \in \mathcal{M} .
\end{array}\right.
\end{equation}
The solution to \eqref{eq:optimal-HJB-equation} depends on the appropriately chosen exploration covariance $\Sigma$. When $\Sigma\equiv0$, the additional diffusion term vanishes, and the equation reduces to the HJB equation associated with the original constrained mean–variance problem studied by \citet{hu2022constrained}.

We now specify the following state-dependent variance function for exploration:
$$\Sigma(t, x, i) = \Psi(t, i)(x-w)^2,$$
where $\Psi: [0, T] \times \mathcal{M} \rightarrow \mathbb{S}^m_{++}$ is any given matrix-valued measurable function.

For later use, define $\mu^{-}(t, i)$ and $\mu^{+}(t, i)$ by
\begin{equation}\label{eq:mu-gamma12}
\begin{aligned}
\mu^{-}(t, i):=&\argmin\limits_{a \in \mathbb{R}_{+}^m}\left\{\frac{1}{2} a^{\top} \sigma(t, i) \sigma(t, i)^{\top} a - a^{\top} \sigma(t, i) \rho(t, i) \right\},\\
\mu^{+}(t, i):=&\argmin\limits_{a \in \mathbb{R}_{+}^m}\left\{\frac{1}{2} a^{\top} \sigma(t, i) \sigma(t, i)^{\top} a + a^{\top} \sigma(t, i) \rho(t, i) \right\}.
\end{aligned}
\end{equation}
Next, we are going to verify that
\begin{equation*}
\tilde{V}^{-}(t, x, i; w) = \tilde{P}^{-}(t, i)(x-w)^2 - (w-z)^2,
\end{equation*}
is the solution of the HJB equation \eqref{eq:optimal-HJB-equation} strictly within the region $\Gamma_-:=\{(t, x)\mid x-w<0\}$, and 
\begin{equation*}
\tilde{V}^{+}(t, x, i; w) = \tilde{P}^{+}(t, i)(x-w)^2 - (w-z)^2,
\end{equation*}
is the solution of the HJB equation \eqref{eq:optimal-HJB-equation} strictly within the region $\Gamma_+:=\{(t, x)\mid x-w \geq 0\}$, respectively.

We first consider the region $\Gamma_-$. 
The partial derivatives of $\tilde{V}^{-}(t, x, i; w)$ are given by:
\begin{equation}\label{eq:partial-V}
\begin{aligned}
\frac{\partial }{\partial x}\tilde{V}^{-}(t, x, i; w) &= 2 \tilde{P}^{-}(t, i)(x-w), \quad \frac{\partial^2 }{\partial x^2}\tilde{V}^{-}(t, x, i; w) = 2 \tilde{P}^{-}(t, i)\\
\frac{\partial }{\partial t}\tilde{V}^{-}(t, x, i; w) &= \frac{\partial}{\partial t}\tilde{P}^{-}(t, i)(x-w)^2.
\end{aligned}
\end{equation}
Applying Lemma \ref{lem:scaling-lemma} and the definition of $\mu^{-}(t, i)$ in \eqref{eq:mu-gamma12}, we obtain 
$$\begin{aligned}
& \min_{a \in \mathbb{R}_+^m} \left\{\frac{1}{2}a^{\top} \sigma(t, i) \sigma(t, i)^{\top} a + (x-w)a^{\top} \sigma(t, i) \rho(t, i)\right\} \\
= & (x-w)^2 \min_{\mu \in \mathbb{R}_+^m} \left\{ \frac{1}{2} \mu^{\top}(t, i) \sigma(t, i) \sigma(t, i)^{\top} \mu(t, i) - \mu^{\top}(t, i) \sigma(t, i) \rho(t, i) \right\}\\
=&(x-w)^2 \bigg[\frac{1}{2} \mu^{-}(t, i)^{\top} \sigma(t, i) \sigma(t, i)^{\top} \mu^{-}(t, i)-\mu^{-}(t, i)^{\top} \sigma(t, i) \rho(t, i)\bigg].
\end{aligned}$$
Substituting \eqref{eq:partial-V} into the left-hand side (LHS) of \eqref{eq:optimal-HJB-equation} and noting $\Sigma(t, x, i) = \Psi(t, i)(x-w)^2$, yields
$$\begin{aligned}
\text{LHS} = & \frac{\partial}{\partial t}\tilde{P}^{-}(t, i)(x-w)^2 + \sum_{j=1}^L q_{ij} \left[\tilde{P}^{-}(t, j)(x-w)^2- (w-z)^2\right]\\
&+ \tilde{P}^{-}(t, i)\operatorname{tr}\left\{\Sigma(t, x, i) \sigma(t, i) \sigma(t, i)^{\top}\right\}\\
&+ 2 \tilde{P}^{-}(t, i)\min_{a \in \mathbb{R}_+^m} \left\{\frac{1}{2}a^{\top} \sigma(t, i) \sigma(t, i)^{\top} a + (x-w)a^{\top} \sigma(t, i) \rho(t, i) \right\}\\
= & (x-w)^2 \left[\frac{\partial}{\partial t}\tilde{P}^{-}(t, i) + \sum_{j=1}^L q_{ij} \tilde{P}^{-}(t, j) \right] + \tilde{P}^{-}(t, i) \operatorname{tr}\left\{ \sigma(t, i) \sigma(t, i)^{\top} \Psi(t, i) (x-w)^2 \right\} \\
& + 2\tilde{P}^{-}(t, i) (x-w)^2 \left[\frac{1}{2} \mu^{-}(t, i)^{\top} \sigma(t, i) \sigma(t, i)^{\top} \mu^{-}(t, i)-\mu^{-}(t, i)^{\top} \sigma(t, i) \rho(t, i)\right]\\
=&\bigg\{\frac{\partial}{\partial t}\tilde{P}^{-}(t, i) + \sum\limits_{j=1}^L q_{ij} \tilde{P}^{-}(t, j) + \tilde{P}^{-}(t, i) \Big[\mu^{-}(t, i)^{\top} \sigma(t, i) \sigma(t, i)^{\top} \mu^{-}(t, i) - 2\mu^{-}(t, i)^{\top} \sigma(t, i) \rho(t, i)\\ &+ \operatorname{tr}(\Psi(t, i) \sigma(t, i) \sigma(t, i)^{\top}) \Big]\bigg\}(x-w)^2.
\end{aligned}$$
Since $(x-w)^2>0$ in $\Gamma_{-}$, the HJB equation holds in this region if and only if $\tilde{P}^{-}(t, i)$ solves the following coupled linear ODE system:
\begin{equation}\label{eq:tilde_P_in_region1}
\begin{cases}
& \frac{\partial}{\partial t}\tilde{P}^{-}(t, i)+\sum\limits_{j=1}^L q_{ij} \tilde{P}^{-}(t, j)+ \tilde{P}^{-}(t, i) \Big[\mu^{-}(t, i)^{\top} \sigma(t, i) \sigma(t, i)^{\top} \mu^{-}(t, i) \\
&\quad - 2\mu^{-}(t, i)^{\top} \sigma(t, i) \rho(t, i) + \operatorname{tr}(\Psi(t, i) \sigma(t, i) \sigma(t, i)^{\top}) \Big]  = 0, \\
&\tilde{P}^{-}(T, i) = 1, \quad i\in \mathcal{M}.
\end{cases}
\end{equation}
An analogous derivation applies to region $\Gamma_+$ yields the corresponding system for $\tilde{P}^{+}(t, i)$:
\begin{equation}\label{eq:tilde_P_in_region2}
\begin{cases}
& \frac{\partial}{\partial t}\tilde{P}^{+}(t, i)+\sum\limits_{j=1}^L q_{ij} \tilde{P}^{+}(t, j) + \tilde{P}^{+}(t, i) \Big[\mu^{+}(t, i)^{\top} \sigma(t, i) \sigma(t, i)^{\top} \mu^{+}(t, i)\\
& \quad +2\mu^{+}(t, i)^{\top} \sigma(t, i) \rho(t, i) + \operatorname{tr}(\Psi(t, i) \sigma(t, i) \sigma(t, i)^{\top}) \Big]  = 0, \\
&\tilde{P}^{+}(T, i) = 1, \quad i\in \mathcal{M}.
\end{cases}
\end{equation}
By the Feynman–Kac formula, this ODE system admits a unique solution, given by the following representations:
\begin{equation}\label{eq:p-t-gamma12}
\begin{aligned}
\tilde{P}^{-}(t, i) & :=\mathbb{E}\bigg[\exp \bigg(\int_t^T[(\mu^{-}_s)^{\top} \sigma_s \sigma_s^{\top} \mu^{-}_s - 2(\sigma_s \rho_s)^{\top} \mu^{-}_s + \operatorname{tr}\{\Psi_s \sigma_s \sigma_s^{\top}\}] d s\bigg) | \alpha_t=i\bigg], \\
\tilde{P}^{+}(t, i) & :=\mathbb{E}\bigg[\exp \bigg(\int_t^T[(\mu^{+}_s)^{\top} \sigma_s \sigma_s^{\top} \mu^{+}_s + 2(\sigma_s \rho_s)^{\top} \mu^{+}_s + \operatorname{tr}\{\Psi_s \sigma_s \sigma_s^{\top}\}] d s\bigg) | \alpha_t=i\bigg].
\end{aligned}
\end{equation}
which also imply $\tilde{P}^{-}(t, i)>0$ and $\tilde{P}^{+}(t, i)>0$.
Here, for brevity, we write $\mu^{-}_s = \mu^{-}(s, \alpha_s)$, $\mu^{+}_s = \mu^{+}(s, \alpha_s)$, $\sigma_s = \sigma(s, \alpha_s)$, $\rho_s = \rho(s, \alpha_s)$, and $\Psi_s = \Psi(s, \alpha_s)$.

Substituting \eqref{eq:tilde_P_in_region1} and \eqref{eq:tilde_P_in_region2} into \eqref{eq:optimal-HJB-equation}, respectively, yields
$\text{LHS} = 0$, in both $\Gamma_-$ and $\Gamma_+$.
Thus, the corresponding quadratic functions satisfy the HJB equation classically in their respective regions. Combining these two solutions, we construct the following value function of the auxiliary problem \eqref{eq:auxiliary-problem}:
\begin{equation}\label{eq:optimal_value_function}
\tilde{V}(t, x, i ; w)= \begin{cases}\tilde{P}^{-}(t, i)(x-w)^2-(w-z)^2, & \text { if } x-w<0, \\
\tilde{P}^{+}(t, i)(x-w)^2-(w-z)^2, & \text { if } x-w>0, \\
-(w-z)^2, & \text { if } x-w=0,\end{cases}
\end{equation}
where $\tilde{P}^{-}(t, i)$ and $\tilde{P}^{+}(t, i)$ are defined in \eqref{eq:p-t-gamma12}.

We next determine the corresponding minimizing control. In the region $\Gamma_-$, the minimization term in \eqref{eq:optimal-HJB-equation} becomes
\begin{equation*}
\begin{aligned}
   &\min\limits_{a \in \mathbb{R}_{+}^m}\left(\frac{\partial}{\partial x} \tilde{V}(t, x, i ; w) a^{\top} \sigma(t, i) \rho(t, i)+\frac{1}{2} \frac{\partial^2}{\partial x^2} \tilde{V}(t, x, i ; w) a^{\top} \sigma(t, i) \sigma(t, i)^{\top} a\right) \\
   &=2 \tilde{P}^{-}(t, i) \min _{a \in \mathbb{R}_{+}^m}\left\{\frac{1}{2} a^{\top} \sigma(t, i) \sigma(t, i)^{\top} a+(x-w) a^{\top} \sigma(t, i) \rho(t, i)\right\}.
\end{aligned}
\end{equation*}
An analogous expression holds in $\Gamma_{+}$, with $\tilde{P}^{-}(t, i)$ replaced by $\tilde{P}^{+}(t, i)$. Since $\tilde{P}^{-}(t, i)>0$ and $\tilde{P}^{+}(t, i)>0$,
 Lemma \ref{lem:scaling-lemma} yields the minimizer
\begin{equation}\label{eq:optimal-control}
\begin{aligned}
a^*(t, x, i;w)
=\begin{cases}
-(x-w)\mu^{-}(t, i), &\text{ if } x-w<0,\\
(x-w)\mu^{+}(t, i), &\text{ if } x-w\geq0,
\end{cases}
\end{aligned}
\end{equation}
where $\mu^{-}(t, i)$ and $\mu^{+}(t, i)$ are defined by \eqref{eq:mu-gamma12}. Notably, $a^*$ coincides with the theoretical optimal control of the target problem derived by \citet{hu2022constrained}.

The following theorem establishes the continuity and viscosity properties of \eqref{eq:optimal_value_function} at $x=w$, together with its optimality.

\begin{theorem}\label{thm:state-dependent}
Let $\Psi:[0, T] \times \mathcal{M} \rightarrow \mathbb{S}_{++}^m$ be a bounded, continuous matrix-valued function, and let $\Pi$ be the admissible policy space defined for the auxiliary problem \eqref{eq:auxiliary-problem}, with the exploration covariance function specified by $\Sigma(t, x, i)=\Psi(t, i)(x-w)^2$. Then, $\tilde{V}(t, x, i ; w)$ defined in \eqref{eq:optimal_value_function} is the optimal value function of the auxiliary problem \eqref{eq:auxiliary-problem} and a unique viscosity solution of the associated HJB equation \eqref{eq:optimal-HJB-equation}. 
Moreover, the auxiliary problem's \eqref{eq:auxiliary-problem} optimal control is:
\begin{equation}\label{eq:optimal_policy}
\pi^*(t, x, i ; w)= \begin{cases}\mathcal{N}\left(-(x-w) \mu^{-}(t, i), \Psi(t, i)(x-w)^2\right), & \text { if } x-w<0, \\
\mathcal{N}\left((x-w) \mu^{+}(t, i), \Psi(t, i)(x-w)^2\right), & \text { if } x-w \geq 0.\end{cases}
\end{equation}
Its mean $a^*(t, x, i; w):=\int u \pi^*(u \mid t, x, i; w) d u$ is given by \eqref{eq:optimal-control} and coincides exactly with the optimal control of the target problem \eqref{eq:optimal-control-problem}. 
\end{theorem}

\begin{proof}
We divide the proof into three steps: verify that the proposed value function satisfies the HJB equation \eqref{eq:optimal-HJB-equation} in each region, examine its regularity and viscosity properties at the interface $\Gamma_0:=\{(t, x)\mid x-w =0\}$, establish the optimality of $\tilde{V}(t, x, i; w)$ and the corresponding Gaussian policy.
        
\noindent{\it 1. Local Verification within $\Gamma_-$ and $\Gamma_+$.} Consider the region $\Gamma_-$, we have
\begin{equation}\label{eq:p-t-gamma12-final}
\tilde{V}(t, x, i; w) = \tilde{P}^{-}(t, i)(x-w)^2 - (w-z)^2, 
\end{equation} 
where $\tilde{P}^{-}>0$.
By Lemma \ref{lem:scaling-lemma} and the definition of $\mu^{-}(t, i)$ in \eqref{eq:mu-gamma12}, we obtain 
\begin{equation}\label{eq:min_mu}
\begin{aligned}
\min _{a \in \mathbb{R}_{+}^m} & \left\{\frac{1}{2} a^{\top} \sigma(t, i) \sigma(t, i)^{\top} a+(x-w) a^{\top} \sigma(t, i) \rho(t, i)\right\} \\
= & (x-w)^2\left\{\frac{1}{2}\mu^{-}(t, i)^{\top} \sigma(t, i) \sigma(t, i)^{\top} \mu^{-}(t, i)-\mu^{-}(t, i)^{\top} \sigma(t, i) \rho(t, i)\right\} .
\end{aligned}
\end{equation}
Substituting \eqref{eq:min_mu} and $\Sigma(t, x, i) = \Psi(t, i)(x-w)^2$ into \eqref{eq:optimal-HJB-equation}, yields
$$\begin{aligned}
\text{LHS}&=(x-w)^2  \bigg\{\frac{\partial }{\partial t}\tilde{P}^{-}(t, i)+\sum_{j=1}^L q_{i j} \tilde{P}^{-}(t, j) +\tilde{P}^{-}(t, i)\Big[\mu^{-}(t, i)^{\top} \sigma(t, i) \sigma(t, i)^{\top} \mu^{-}(t, i)\\
&\quad -2\mu^{-}(t, i)^{\top} \sigma(t, i) \rho(t, i) +\operatorname{tr}(\Psi(t, i) \sigma(t, i) \sigma(t, i)^{\top})\Big]\bigg\}\\
&=0.
\end{aligned}$$
where the last equality follows from the coupled linear ODE system \eqref{eq:tilde_P_in_region1}. An analogous argument applies to the interior of $\Gamma_{+}$. Hence, $\tilde{V}(t, x, i; w)$ satisfies \eqref{eq:optimal-HJB-equation} classically in the interiors of $\Gamma_{-}$ and $\Gamma_{+}$, respectively. Moreover, the terminal conditions $\tilde{P}^{-}(T, i)=1$ and $\tilde{P}^{+}(T, i)=1$ ensure that
$$
\tilde{V}(T, x, i ; w)=(x-w)^2-(w-z)^2 .
$$
                
\noindent{\it 2. Regularity and viscosity properties at the interface.} To ensure $\tilde{V}(t, x, i; w)$ is a global solution, we examine the interface $\Gamma_0$.
\begin{itemize}
\item Continuity: By the definition of $\tilde{V}(t, x, i ; w)$ in \eqref{eq:optimal_value_function},
$$
\lim _{x \uparrow w} \tilde{V}(t, x, i ; w)=\lim _{x \downarrow w} \tilde{V}(t, x, i ; w)=-(w-z)^2 .
$$
Thus, $\tilde{V}(t, x, i; w)$ is continuous across the entire domain.
\item Smoothness: The first-order spatial derivative of $\tilde{V}(t, x, i ; w)$ is given by
$$
\frac{\partial}{\partial x} \tilde{V}(t, x, i ; w)= \begin{cases}2 \tilde{P}^{-}(t, i)(x-w), & x-w<0, \\ 2 \tilde{P}^{+}(t, i)(x-w), & x - w \geq 0.\end{cases}
$$ 
Taking the limit as $x \to w$, we find that $\lim\limits_{x \uparrow w} \frac{\partial }{\partial x}\tilde{V}(t, x, i; w)= \lim\limits_{x \downarrow w} \frac{\partial }{\partial x}\tilde{V}(t, x, i; w) = 0$.
\end{itemize}

Since the time derivative also tends to zero at $x=w$,$\tilde{V}(t, x, i ; w)$ is continuously differentiable across $\Gamma_0$ under the stated temporal regularity conditions. However, its second-order spatial derivatives satisfy
$$
\lim _{x \uparrow w} \frac{\partial^2 }{\partial x^2}\tilde{V}(t, x, i ; w)=2 \tilde{P}^{-}(t, i), \quad \lim _{x \downarrow w} \frac{\partial^2}{\partial x^2}\tilde{V}(t, x, i ; w)=2 \tilde{P}^{+}(t, i),
$$
which need not coincide. Thus, $\tilde{V}(t, x, i ; w)$ is not necessarily a globally $C^{1,2}$ function.

It remains to verify the viscosity conditions at $\Gamma_0$. We consider the backward Bellman formulation of \eqref{eq:optimal-HJB-equation}. Let a smooth test function $\varphi$ touch the $i$-th component of $\tilde{V}(t, x, i ; w)$ from above at an interior point $(t_0, w)$. Since $\tilde{V}(t, w, j)=-(w-z)^2$ for every $j \in \mathcal{M}$, we have
\begin{equation}\label{eq:identities}
\varphi_t(t_0, w)=0, \quad \varphi_x(t_0, w)=0, \quad \sum_{j=1}^L q_{i j} \tilde{V}(t_0, w, j)=0 .
\end{equation}
Moreover, the upper contact implies
$$
\varphi_{x x}(t_0, w) \geq 2 \max \left\{\tilde{P}^{-}(t_0, i), \tilde{P}^{+}(t_0, i)\right\}>0 .
$$
Since $\Sigma(t_0, w, i)=0$, the backward Bellman operator evaluated at this contact point reduces to
$$
\sup _{a \in \mathbb{R}_{+}^m}\left\{-\frac{1}{2} \varphi_{x x}(t_0, w) a^{\top} \sigma(t_0, i) \sigma(t_0, i)^{\top} a\right\}=0,
$$
where the supremum is attained at $a=0$. Hence, the viscosity subsolution condition holds.

Next, suppose that $\varphi$ touches the $i$-th component of $\tilde{V}(t, x, i ; w)$ from below at $(t_0, w)$, the same identities in \eqref{eq:identities} hold. In this case, the lower contact implies
$$
\varphi_{x x}\left(t_0, w\right) \leq 2 \min \left\{\tilde{P}^{-}\left(t_0, i\right), \tilde{P}^{+}\left(t_0, i\right)\right\} .
$$
To account for the unbounded control set, let $K_R=\left\{a \in \mathbb{R}_{+}^m:\|a\| \leq R\right\}$, where $R>0$. Since $0 \in K_R$,
$$
\sup _{a \in K_R}\left\{-\frac{1}{2} \varphi_{x x}(t_0, w) a^{\top} \sigma(t_0, i) \sigma(t_0, i)^{\top} a\right\} \geq 0 .
$$
Thus, the viscosity supersolution condition also holds. Together with the local HJB verification in Step 1 and the terminal condition, this establishes the viscosity solution property of $\tilde{V}(t, x, i ; w)$ under the standard backward Bellman convention.

\noindent{\it 3. Verification of optimality.} 
Suppose that $U(t, x, i ; w)$ is another viscosity solution of \eqref{eq:optimal-HJB-equation} equation with the same terminal condition and of the form
$$
U(t, x, i ; w)= \begin{cases}\widehat{P}^{-}(t, i)(x-w)^2-(w-z)^2, & x-w<0, \\ \widehat{P}^{+}(t, i)(x-w)^2-(w-z)^2, & x -w\geq 0,\end{cases}
$$
where $\widehat{P}^{ \pm}(\cdot, i) \in C^1([0, T])$ and $\widehat{P}^{ \pm}(t, i)>0$. Since $U$ is $C^{1,2}$ in the interior of each region, it can be used as a test function touching itself from both above and below. Moreover, the positivity of $\widehat{P}^{ \pm}(t, i)$ ensures that the Hamiltonian is finite. Hence, the viscosity subsolution and supersolution conditions imply that $U$ satisfies the HJB equation \eqref{eq:optimal-HJB-equation} pointwise in each region. By Lemma \ref{lem:scaling-lemma}, the corresponding Hamiltonian minimizers are given by \eqref{eq:optimal-control}, independently of $\widehat{P}^{ \pm}(t, i)$. Substituting the piecewise quadratic form of $U$ into the HJB equation, we obtain the same coupled linear ODE systems for $\widehat{P}^{ \pm}(t, i)$ as those satisfied by $\tilde{P}^{ \pm}(t, i)$, with the terminal conditions $\widehat{P}^{ \pm}(T, i)=1$.

Let $D^{ \pm}(t, i):=\widehat{P}^{ \pm}(t, i)-\tilde{P}^{\pm}(t, i)$.
Subtracting the corresponding ODE systems yields
$$
\left\{\begin{array}{l}
\frac{\partial D^{ \pm}}{\partial t}(t, i)+\sum_{j=1}^L q_{i j} D^{ \pm}(t, j)+\kappa^{ \pm}(t, i) D^{ \pm}(t, i)=0, \\
D^{ \pm}(T, i)=0, \quad i \in \mathcal{M},
\end{array}\right.
$$
where
$$\begin{aligned}
\kappa^{-}(t, i)= & \mu^{-}(t, i)^{\top} \sigma(t, i) \sigma(t, i)^{\top} \mu^{-}(t, i) \\
& -2\mu^{-}(t, i)^{\top} \sigma(t, i) \rho(t, i)+\operatorname{tr}\left\{\Psi(t, i) \sigma(t, i) \sigma(t, i)^{\top}\right\}, \\
\kappa^{+}(t, i)= & \mu^{+}(t, i)^{\top} \sigma(t, i) \sigma(t, i)^{\top} \mu^{+}(t, i) \\
& +2\mu^{+}(t, i)^{\top} \sigma(t, i) \rho(t, i)+\operatorname{tr}\left\{\Psi(t, i) \sigma(t, i) \sigma(t, i)^{\top}\right\} .
\end{aligned}$$
By the uniqueness of solutions to finite-dimensional linear ODE systems with the prescribed terminal conditions,
$D^{ \pm}(t, i)=0 .$
Hence,
$\widehat{P}^{ \pm}(t, i)=\tilde{P}^{ \pm}(t, i),$
which implies
$U(t, x, i ; w)=\tilde{V}(t, x, i ; w)$
for all $(t, x, i) \in[0, T] \times \mathbb{R} \times \mathcal{M}$. Together with the viscosity solution property established in Steps 1–2, this proves that $\tilde{V}$ is the unique viscosity solution of \eqref{eq:optimal-HJB-equation}.
\end{proof}

\begin{remark}
Theorem \ref{thm:state-dependent} provides several important insights into the auxiliary exploratory problem. First, when $\Sigma \equiv 0$, the additional diffusion term vanishes, and \eqref{eq:optimal-HJB-equation} formally reduces to the HJB equation associated with the original constrained mean-variance problem studied by \citet{hu2022constrained}. More importantly, the mean $a^*$ of the optimal Gaussian policy, given by \eqref{eq:optimal-control}, coincides exactly with the optimal control of the corresponding target problem. In particular, the choice of $\Psi(t, i)$ does not affect the optimal mean control, but determines the exploration variance and influences the resulting value function. Furthermore, under state-dependent exploration, both the mean and covariance of $\pi^*(t, x, i ; w)$ vanish as $x \rightarrow w$. Consequently, $x=w$ is an absorbing state for the optimally controlled wealth process. Under the usual well-posedness conditions, the continuity of wealth trajectories and absorption at $x=w$ also prevent state trajectory from crossing the interface between $\Gamma_{-}$and $\Gamma_{+}$.
\end{remark}

Finally, we determine the optimal Lagrange multiplier $w^*$. The optimal cost functional evaluated at $t=0$ is
\begin{equation*}
J(w; x_0, i)= \begin{cases}
\tilde{P}^{-}(0, i)(w-x_0)^2-(w-z)^2, & \text { if } w-x_0 >0, \\
\tilde{P}^{+}(0, i)(w-x_0)^2-(w-z)^2, & \text { if } w -x_0 \leq 0.\end{cases}
\end{equation*}
Under the assumption $z > x_0$, we focus on the domain $w > x_0$. By the first-order condition, we obtain
\begin{equation}\label{eq:optimal-Lagrange-multiplier}
w^* = \operatorname*{argmax}_{w \in \mathbb{R}} J(u^*(\cdot; w); w, i) = \frac{z - x_0 \tilde{P}^{-}(0, i)}{1 - \tilde{P}^{-}(0, i)} = \frac{1-\frac{x_0}{z} \tilde{P}^{-}(0, i)}{1-\tilde{P}^{-}(0, i)} z.
\end{equation}
Given that the solution to the ODE system satisfies $0 <\tilde{P}^{-}(0, i) < 1$ and $z > x_0$, we have $w^* > z > x_0$. This confirms that the optimal wealth process remains in the region $\Gamma_-$ almost surely, thereby justifying the use of $\tilde{P}^{-}$ in the final expression of the optimal policy.

\section{A Continuous-Time RL Algorithm.}\label{sec4}
    
    In this section, we develop a continuous-time reinforcement learning algorithm to learn the optimal control of the target problem \eqref{eq:optimal-control-problem}. Based on the theoretical results presented in the previous section, \eqref{eq:constrained-optimization-problem} is equivalent to finding the optimal control of the auxiliary problem \eqref{eq:auxiliary-problem}. Then, we define the martingale loss function and update our policy within the feasible set using stochastic gradient descent.

	\subsection{Function Parameterization.}\label{subsection_paramtrization}
    
    This subsection develops the function parameterizations used in the proposed learning algorithm. We first characterize the continuous-time structure of the value function and policy, then introduce their piecewise parameterizations on a discrete time grid for numerical implementation.
    
    {\it Continuous-Time Parameterization.} We use a state-dependent variance function $\Sigma(t, x, i)=\Psi_0(x-w)^2$, where $\Psi_0\in \mathbb{S}_{++}^n$ is a fixed matrix. We introduce the following vector functions $\boldsymbol{\phi}^{-}(t, i)$, $\boldsymbol{\phi}^{+}(t, i): [0, T] \times \mathcal{M} \to \mathbb{R}_{+}^m$. Then the policy distribution can be parameterized as:
    \begin{equation}\label{eq:pi-continuous}
        \pi^\phi(u \mid t, x, i) = \mathcal{N}\left( u^\phi(t, x, i; w), \Psi_0(x-w)^2 \right),
    \end{equation}
    where the mean function $u^\phi$ is defined by
    \begin{equation*}\label{eq:mean-control-function}
        u^\phi(t, x, i; w) = \begin{cases}
        -(x-w)\boldsymbol{\phi}^{-}(t, i), & \text{if } x-w < 0, \\ 
        (x-w)\boldsymbol{\phi}^{+}(t, i), & \text{if } x-w \geq 0 .
         \end{cases}
    \end{equation*}
    We also introduce two scalar functions $\theta^{-}(t, i), \theta^{+}(t, i): [0, T] \times \mathcal{M} \to \mathbb{R}$. The value function of $\pi^\phi$ is defined by
    \begin{equation}\label{eq:value-function-expection}
        V^\theta(t, x, i; w) = \mathbb{E} \big[ (x_T^\pi-w)^2 - (w-z)^2 \mid x_t^\pi=x, \alpha_t=i \big].
    \end{equation}
    The piecewise-linear policy \eqref{eq:pi-continuous} provides an explicit parameterization, which motivates the following parameterization of the value function. By the Feynman-Kac formula, the value function \eqref{eq:value-function-expection} satisfies the following linear PDE
    \begin{equation*}\label{eq:parameterize-HJB-equation}
        \left\{\begin{array}{l}
             \frac{\partial}{\partial t} v(t, x, i) 
            + \frac{1}{2} \frac{\partial^2}{\partial x^2} v(t, x, i) u^\phi(t, x, i; w)^{\top} \sigma(t, i) \sigma(t, i)^{\top} u^\phi(t, x, i; w) + \sum\limits_{j=1}^L q_{ij} v(t, x, j)\\
            \,\,+\frac{1}{2} \frac{\partial^2}{\partial x^2} v(t, x, i) \operatorname{tr}\{\Psi_0(x-w)^2 \sigma(t, i) \sigma(t, i)^{\top}\} + \frac{\partial}{\partial x} v(t, x, i) u^\phi(t, x, i; w)^{\top} \sigma(t, i) \rho(t, i)  = 0, \\
            v(T, x, i) = (x-w)^2 - (w-z)^2, \quad \forall i \in \mathcal{M}.
        \end{array}\right.
    \end{equation*}
    Solving it yields 
    \begin{equation}\label{eq:V-continuous}
        V^\theta(t, x, i; w) = \begin{cases}
            \exp(\theta^{-}(t, i))(x-w)^2 - (w-z)^2, & \text{if } x-w < 0, \\
            \exp(\theta^{+}(t, i))(x-w)^2 - (w-z)^2, & \text{if } x-w > 0, \\
            - (w-z)^2, & \text{if } x-w = 0,
        \end{cases}
    \end{equation}
    where the functions $\theta^{-}(t, i)$ and $\theta^{+}(t, i)$ are uniquely determined by
    \begin{equation*}
        \begin{aligned}
            \exp(\theta^{-}(t, i)) &= \mathbb{E}\bigg[\exp \bigg( \int_t^T \big[ (\boldsymbol{\phi}^{-}(s, \alpha_s))^{\top} \sigma_s \sigma_s^{\top} \boldsymbol{\phi}^{-}(s, \alpha_s) - 2(\sigma_s \rho_s)^{\top} \boldsymbol{\phi}^{-}(s, \alpha_s) + \operatorname{tr}\{\Psi_s \sigma_s \sigma_s^{\top}\} \big] d s \bigg) \,\bigg|\, \alpha_t=i \bigg],\\
            \exp(\theta^{+}(t, i)) &= \mathbb{E}\bigg[\exp \bigg( \int_t^T \big[ (\boldsymbol{\phi}^{+}(s, \alpha_s))^{\top} \sigma_s \sigma_s^{\top} \boldsymbol{\phi}^{+}(s, \alpha_s) + 2(\sigma_s \rho_s)^{\top} \boldsymbol{\phi}^{+}(s, \alpha_s) + \operatorname{tr}\{\Psi_s \sigma_s \sigma_s^{\top}\} \big] d s \bigg) \,\bigg|\, \alpha_t=i \bigg].
        \end{aligned}
    \end{equation*}
 
    {\it Temporal Discretization and Piecewise Parameterization}. 
    We partition the time horizon $[0, T]$ into $N$ equidistant sub-intervals with a step size $\Delta t = T/N$, yielding discrete time nodes $t_n = n\Delta t$ for $n \in \{0, \dots, N\}$. Within this discretized framework, we employ a piecewise-constant approximation. We assume that the continuous functions remain time-invariant within each infinitesimal sub-interval $[t_n, t_{n+1})$. Consequently, 
    for each regime $i \in \mathcal{M}$ and time interval $[t_n, t_{n+1}), n \in \{0, \dots, N-1\}$, we can define the parameter sets
    \begin{equation*}
        \boldsymbol{\Theta}^\phi := \{ \boldsymbol{\phi}_{i,n}^{-}, \boldsymbol{\phi}_{i,n}^{+} \}, \quad \boldsymbol{\Theta}^\theta := \{ \theta_{i,n}^{-}, \theta_{i,n}^{+} \},
    \end{equation*}
    where $\boldsymbol{\phi}_{i, n}^{ \pm}:=\boldsymbol{\phi}^{ \pm}\left(t_n, i\right) \in \mathbb{R}_{+}^m$ and $\theta_{i, n}^{ \pm}:=\theta^{ \pm}\left(t_n, i\right) \in \mathbb{R}$, with the superscript $\pm$ indicating both the $+$ and $-$ cases.
    To ensure the value function at the terminal time is consistent with the target cost, i.e., $V^\theta(T, x, \cdot) = (x-w)^2 - (w-z)^2$, we enforce the boundary constraint at the final grid node: $\theta_{i,N}^{-} = \theta_{i,N}^{+} = 0$, $\forall i \in \mathcal{M}$.

	\subsection{Policy Iteration.}
	In this section, we derive a continuous-time reinforcement learning framework without market parameters, so that our algorithm can learn directly from trajectory samples.

    {\it The Definition of Martingale Loss.} To estimate the value function $V(t, x, i; w)$, we employ the martingale loss approach proposed in \citet{jia2021policy}. The martingale loss $L(\theta)$ is defined as:
    \begin{equation*}
        \begin{aligned}
            L(\theta) :&= \mathbb{E} \left[ \int_0^T \left( V^\theta(t, X_t, \alpha_t; w) - V^\phi(T, X_T, \alpha_T ; w)\right)^2 dt \right] \\
            &= \mathbb{E} \left[ \int_0^T \left( V^\theta(t, X_t, \alpha_t; w) - \left[ (X_T-w)^2 - (w-z)^2 \right] \right)^2 dt \right].
        \end{aligned}
    \end{equation*}
    Assuming that the parameters $\{\theta_{i,n}^{-}\}$ and $\{\theta_{i,n}^{+}\}$ are constant within each time interval $[t_n, t_{n+1})$, the partial derivatives of $L(\theta)$ with respect to the parameters are derived as follows:\footnote{
    For the Critic parameters, the gradient is updated using the terminal state $X_T$ of the entire trajectory, which eliminates bootstrapping bias and ensures that the parameterized model converges to the true value function of the current policy:
    $$\begin{aligned}
        &\nabla_{\theta_{i,n}^{\pm}} L(\theta) = \mathbb{E}\left[\int_0^T 2\left(V^\theta(t, X_t, \alpha_t ; w)-V(T, X_T, \alpha_T ; w)\right) \frac{\partial V^\theta(t, X_t, \alpha_t ; w)}{\partial \theta_{i, n}^{\pm}} d t\right]\\
        =& 2 \mathbb{E} \left[ \int_0^T \left( \exp(\theta_{i,n}^{\pm})(X_t-w)^2 - (X_T-w)^2 \right) \cdot \exp(\theta_{i,n}^{\pm})(X_t-w)^2 \cdot \mathbb{I}_{\{t \in[t_n, t_{n+1})\}} \cdot \mathbb{I}_{\{\alpha_t=i, X_t-w < 0\}}dt \right],
    \end{aligned}$$
    where $\mathbb{I}$ denotes the indicator function.}
    $$\begin{aligned}
        \frac{\partial L(\theta)}{\partial \theta_{i, n}^{-}} &=
        2\Delta t \mathbb{E} \bigg[\mathbb{I}_{\{\alpha_n=i, x_n<w\}}\big(\exp(\theta_{i, n}^{-})(x_n-w)^2 - (x_N-w)^2\big) \cdot \exp(\theta_{i,n}^{-})(x_n-w)^2 \bigg],  \\
        \frac{\partial L(\theta)}{\partial \theta_{i, n}^{+}} &= 2\Delta t \mathbb{E} \bigg[\mathbb{I}_{\{\alpha_n=i, x_n\geq w\}}\big(\exp(\theta_{i, n}^{+})(x_n-w)^2 - (x_N-w)^2\big) \cdot \exp(\theta_{i,n}^{+})(x_n-w)^2 \bigg].
    \end{aligned}$$
   
    {\it Parameter Iteration and Update.} We use the mini-batch stochastic gradient descent (SGD) method to optimize the martingale loss. Let $\mathcal{B}$ denote a mini-batch of trajectories. At each time step $t_n$, let $\mathcal{B}_i$ denote the subset of trajectories for which the market regime is $\alpha_{t_n} = i$. We further partition $\mathcal{B}_i$ into two disjoint subsets based on the wealth regions:
    $$
      \mathcal{B}_{i,n}^{-} := \{ m \in \mathcal{B}_i \mid x_n^{(m)} < w \}, \quad \mathcal{B}_{i,n}^{+} := \{ m \in \mathcal{B}_i \mid x_n^{(m)} \geq w \}.
    $$
    The empirical gradients for the parameters at time step $t_n$ are calculated by
    $$\begin{aligned}
        \frac{\partial L(\theta)}{\partial \theta_{i, n}^{-}} &= \frac{1}{|\mathcal{B}_{i,n}^{-}|} \sum_{m \in \mathcal{B}_{i,n}^{-}} \left[ 2\left(\exp(\theta_{i, n}^{-})(x_n^{(m)}-w)^2 - (x_N^{(m)}-w)^2\right)\cdot \exp(\theta_{i, n}^{-})(x_n^{(m)}-w)^2 \Delta t \right], \\
        \frac{\partial L(\theta)}{\partial \theta_{i, n}^{+}} &= \frac{1}{|\mathcal{B}_{i,n}^{+}|} \sum_{m \in \mathcal{B}_{i,n}^{+}} \left[ 2\left(\exp(\theta_{i, n}^{+})(x_n^{(m)}-w)^2 - (x_N^{(m)}-w)^2\right)\cdot \exp(\theta_{i, n}^{+})(x_n^{(m)}-w)^2 \Delta t \right].
    \end{aligned}$$
    And we divide by the total number of regime-matching trajectories $|\mathcal{B}_{i,n}^{-}|$ to maintain an unbiased estimate of the conditional expectation over the batch. Then the parameters are updated iteratively via 
    \begin{equation}\label{eq:MSE-theta-updated}
        \begin{aligned}
            \theta_{i, n}^{-} \leftarrow \theta_{i, n}^{-} - \eta_\theta \frac{\partial L(\theta)}{\partial \theta_{i, n}^{-}},\quad
            \theta_{i, n}^{+} \leftarrow \theta_{i, n}^{+} - \eta_\theta \frac{\partial L(\theta)}{\partial \theta_{i, n}^{+}}.
        \end{aligned}
    \end{equation}
    where $\eta_\theta > 0$ is the learning rate.

\subsection{Policy Gradient Optimization.}

    Let $\pi^\phi$ be the current parameterized policy. Let $(x(\cdot), \alpha(\cdot), u(\cdot))$ be the state-regime-action trajectory of \eqref{eq:wealth-process} generated under the randomized policy $\pi^\phi(\cdot)$. According to the continuous-time policy gradient theorem, the gradient of the objective function with respect to the policy parameters can be represented using the stochastic TD error:
    $$
        \frac{\partial}{\partial \phi} \tilde{J}(\pi^\phi(\cdot) ; w) = \mathbb{E} \left[\int_0^T \left( \frac{\partial}{\partial \phi} \log \pi^\phi(u_t \mid t, x_t, \alpha_t) \right) d V^\phi(t, x_t, \alpha_t ; w)\right].
    $$
    The partial derivatives of $\tilde{J}(\pi^\phi(\cdot) ; w)$ with respect to the parameters are derived as follows:
    $$\begin{aligned}
        & \frac{\partial}{\partial \phi_{i,n}^{-}} \tilde{J}\left(\pi^\phi(\cdot) ; w\right) \approx \mathbb{E}\left[ \mathbb{I}_{\{\alpha_n=i, x_n<w\}} \big\{-\Psi_0^{-1} \left( \frac{u_n}{x_n-w} + \phi_{i,n}^{-} \right)\big\} \delta_n^\phi \right],  \\
        & \frac{\partial}{\partial \phi_{i,n}^{+}} \tilde{J}\left(\pi^\phi(\cdot) ; w\right) \approx \mathbb{E}\left[ \mathbb{I}_{\{\alpha_n=i, x_n\geq w\}} \big\{ \Psi_0^{-1} \left( \frac{u_n}{x_n-w} + \phi_{i,n}^{+} \right)\big\} \delta_n^\phi \right],
    \end{aligned}$$
    where $\delta_n^\phi = V^\phi\left(t_{n+1}, x_{n+1}, \alpha_{n+1}; w\right) - V^\phi\left(t_n, x_n, \alpha_n; w\right)$ is the empirical TD error evaluated by the Critic network. In the Actor–Critic implementation, the increment of the true policy function, $dV^\phi$, is replaced by the corresponding increment of the parameterized Critic, $V^\theta$. 
    Given the discretized trajectory $\{(t_n, x_n, \alpha_n, u_n)\}_{n=0}^{N-1}$, the empirical gradients for the parameters at time step $t_n$ are calculated by:
    $$\begin{aligned}
        \frac{\partial}{\partial \phi_{i,n}^{-}} \tilde{J}\left(\pi^\phi(\cdot) ; w\right) &\approx \frac{1}{|\mathcal{B}_{i,n}^{-}|} \sum_{m \in \mathcal{B}_{i,n}^{-}} \left[ -\Psi_0^{-1} \left( \frac{u_n^{(m)}}{x_n^{(m)}-w} + \phi_{i,n}^{-} \right) \delta_{m,n}^{\theta^-} \right], \\
        \frac{\partial}{\partial \phi_{i,n}^{+}} \tilde{J}\left(\pi^\phi(\cdot) ; w\right) &\approx \frac{1}{|\mathcal{B}_{i,n}^{+}|} \sum_{m \in \mathcal{B}_{i,n}^{+}} \left[\Psi_0^{-1} \left( \frac{u_n^{(m)}}{x_n^{(m)}-w} + \phi_{i,n}^{+} \right) \delta_{m,n}^{\theta^+} \right],
    \end{aligned}$$
    where $\delta_{m, n}^\theta=V^\theta(t_{n+1}, x_{n+1}^{(m)}, \alpha_{n+1}^{(m)} ; w)-V^\theta(t_n, x_n^{(m)}, \alpha_n^{(m)} ; w)$.
    And update parameters $\{\phi_{i, n}^{-}\}$ and $\{\phi_{i, n}^{+}\}$ by:
   \begin{equation}\label{eq:MSE-phi-update}
        \begin{aligned}
            \phi_{i, n}^{-} \leftarrow \max\bigg\{0, \phi_{i, n}^{-}-\eta_\phi \frac{\partial}{\partial \phi_{i,n}^{-}} \tilde{J}\left(\pi^\phi(\cdot) ; w\right)\bigg\},\\
            \phi_{i, n}^{+} \leftarrow \max\bigg\{0, \phi_{i, n}^{+}-\eta_\phi \frac{\partial}{\partial \phi_{i,n}^{+}} \tilde{J}\left(\pi^\phi(\cdot) ; w\right)\bigg\}.
        \end{aligned}
    \end{equation}
    
    Considering the nonnegative integral constraint that requires $\pi^\phi$ to have a non-negative mean, we ensure $\phi_{i,n}^{-} \geq 0$ and $\phi_{i,n}^{+} \geq 0$ by projecting them back to the feasible set $\mathbb{R}_{+}^m$ after each gradient descent step. This is essentially the proximal gradient method for constrained optimization problems \citep{beck2017First}. In our framework, the proximal gradient method is guaranteed to converge because the exponentiated cost functional in \eqref{eq:V-continuous} is convex with respect to the trainable parameters $\phi$.

    Finally, we design the update rule for the moving reference target $w$ to dynamically estimate the optimal multiplier $w^*$. Notice that $w^*$ is the unique root of the terminal wealth expectation condition:
    \begin{equation*}\label{eq:w-update-condition}
        \mathbb{E}[x^*(T;w)] - z = 0,
    \end{equation*}
    where $x^*(T;w)$ is the optimal terminal state of the auxiliary problem \eqref{eq:auxiliary-problem}, which shares the same mean path as the target problem \eqref{eq:optimal-control-problem}. Consequently, we employ the Robbins-Monro stochastic approximation method to iteratively update $w$ at the end of each training episode $n_w$:
    \begin{equation*}\label{eq:w-update-rule}
        w \leftarrow w - \eta_w \left( \frac{1}{M} \sum_{m=1}^M x^{(m)}(T; w) - z \right),
    \end{equation*}
    where $\eta_w > 0$ is the learning rate for the multiplier, and $\frac{1}{M} \sum_{m=1}^M x^{(m)}(T; w)$ is the Monte Carlo empirical mean of the terminal wealth evaluated over $M$ sampled trajectories driven by the current policy $\pi^\phi$.
    
	The analysis above leads to Algorithm \ref{algo_ba2c_modified} for the continuous-time MV portfolio selection problem with a no-shorting constraint for a Markov regime-switching system.
    
    \begin{algorithm}
    \begin{algorithmic}[1]
	\caption{Continuous-Time Portfolio Optimization under Regime-Switching}\label{algo_ba2c_modified}

    	\State \textbf{Input:} Initial parameters $\bm{\Theta}^{\theta}_0$ and $\bm{\Theta}^{\phi}_0$; learning rates $\eta_\theta$, $\eta_\phi$, and $\eta_w$; update frequencies $n_\phi$, $n_\theta$, and $n_w$;
        warm-up iterations $n_{\mathrm{warmup}}$; batch size $M$; target discounted expected return $z$ and initial Lagrange multiplier $w_0$; and terminal constraints $\theta_{i,N}^{-}=\theta_{i,N}^{+}=0$ for all $i\in\mathcal{M}$.
    
        \State Initialize $w\gets w_0$.
		
	  \\\hrulefill
		
         \For{$c=1$ \textbf{to} $n_{\mathrm{warmup}}$}
            \State Initialize time $t_0 = 0$, state $x_{t_0}^{(m)} = x_0$, and initial regime $\alpha_{t_0}^{(m)}=i$.
            \State Sample $M$ trajectories by executing initial policy $\pi(\cdot \mid t_0, x_0, i ; \Theta_0^\phi)$, organized as $\mathcal{D}_c:=\{(x_{t_n}^{(m)}, \alpha_{t_n}^{(m)}, u_{t_n}^{(m)}), n=0,1, \ldots, N\}_{m=1}^M$ (the \textbf{current batch}).
            \State Identify trajectory subsets $\mathcal{B}_i$, $\mathcal{B}_{i,n}^{-}$, $\mathcal{B}_{i,n}^{+}$ and update $\{\theta_{i,n}^{-}\}$ and $\{\theta_{i,n}^{+}\}$ via \eqref{eq:MSE-theta-updated}.
        \EndFor
        
        \State Initialize an empty historical terminal wealth buffer $\mathcal{X}$.

		\For{\texttt{$e = 1$ to $n_\phi$}}
            \State Initialize an empty episode data buffer $\mathcal{D}$ (the \textbf{accumulated buffer}).
            
		    \For{\texttt{$c = 1$ to $n_\theta$}}
                \State Sample $M$ trajectories by executing $\pi^\phi$ via \eqref{eq:pi-continuous}, organized as $\mathcal{D}_c$. 
                
                
                \State Identify subsets $\mathcal{B}_i, \mathcal{B}_{i,n}^{-}, \mathcal{B}_{i,n}^{+}$ from $\mathcal{D}_c$, and update $\{\theta_{i,n}^{-}\}$ and $\{\theta_{i,n}^{+}\}$ via \eqref{eq:MSE-theta-updated}.
                \State Append $\mathcal{D}_c$ to buffer $\mathcal{D}$.
	        \EndFor

            \State Identify subsets $\mathcal{B}_i, \mathcal{B}_{i,n}^{-}, \mathcal{B}_{i,n}^{+}$ from  $\mathcal{D}$, and update $\{\phi_{i, n}^{-}\}$ and $\{\phi_{i, n}^{+}\}$ via \eqref{eq:MSE-phi-update}.
            
	        \State Compute average terminal wealth of $\mathcal{D}$: $\bar{x}_T = \frac{1}{n_\theta} \sum_{m=1}^{n_\theta} x_{t_N}^{(m)}$.
            \State Append $\bar{x}_T$ to the historical buffer $\mathcal{X}$.
	        \If{$(e \bmod n_w) = 0$}
                \State Compute $\hat{x}_T$ as the average of the most recent $n_w$ entries in $\mathcal{X}$.
	            \State Update target multiplier: $w \leftarrow w - \eta_w (\hat{x}_T - z)$.
	        \EndIf
		\EndFor

    \\\hrulefill
    \State \textbf{return}
    Optimal parameters for value function $\bm{\Theta}^{\theta}_E$, policy $\bm{\Theta}^{\phi}_E$, and multiplier $w_E$.
    
    \end{algorithmic}
    \end{algorithm}

\section{Numerical Experiments.}\label{sec5}
        
	In this section, we verify the effectiveness of the Algorithm \ref{algo_ba2c_modified} for solving the mean-variance portfolio selection problem with no-shorting constraints under the Markov regime-switching model through numerical experiments. We consider a financial market with two market regimes (i.e., the “bull” and the “bear” regimes, $\mathcal{M}=\{0,1\}$) and three risky assets ($m=3$). The total investment horizon is set to $T=1.0$ with an initial wealth of $x_0=1.0$. and the continuous-time horizon is partitioned into $N=100$ equidistant steps, corresponding to a time step of $\Delta t = T/N = 0.01$. Under the “bull” regime (Regime 0), the risk-free rate is set to $r_0=0.03$, and the expected return vector of the risky assets is set to $b_0=[0.09, 0.15, 0.19]^\top$; under the “bear” regime (Regime 1), the risk-free interest rate is set to $r_1=0.01$, and the expected return vector of risk assets is $b_1=[0.07, 0.13, 0.17]^\top$. The asset volatility matrices are generated by multiplying the preset diagonal volatility matrix by the Cholesky factor corresponding to the prescribed correlation structure. Specifically, for both regimes, the volatility matrix is given by $\sigma_0=\sigma_1 = \operatorname{diag}(0.20, 0.25, 0.35) \Lambda$, where
    $$
    \Lambda = \begin{pmatrix}
    1 & 0 & 0 \\
    0.5 & \sqrt{3}/2 & 0 \\
    0.5 & \sqrt{3}/6 & \sqrt{6}/3
    \end{pmatrix}
    $$
    is the Cholesky factor.

    The regime transition intensities are strictly governed by a generator matrix (Q-matrix), where $q_{12}=1.0$ and $q_{21}=2.0$. For algorithmic implementation and sample generation, we employ the Euler-Maruyama discretization scheme to simulate the evolution of the controlled wealth process, and use transition probabilities derived from the Markov chain's generator matrix to simulate regime transitions. For each batch, we independently simulate $M=64$ wealth trajectories. At each decision time, actions are sampled from the parameterized policy distribution $\pi^\phi(u \mid t, x, i)$ defined in \eqref{eq:pi-continuous}, with the exploration variance constant set to $\Psi_0=0.2$. 

    The theoretical solution of the target problem \eqref{eq:optimal-control-problem} specializes to
    $$u^*(t, x, i)=\left\{\begin{array}{ll}
    -\phi_i^*\left(x-w^*\right), & x<w^*, \\
    0, & x \geq w^*,
    \end{array} \quad i \in \mathcal{M}=\{0,1\},\right.
    $$
    where $\phi_i^* \in \mathbb{R}_{+}^3$ denotes the optimal policy coefficient for regime $i$, and $w^* \in \mathbb{R}$ is the optimal Lagrange multiplier defined in \eqref{eq:optimal-Lagrange-multiplier}. Define
    $c_i\left(\phi_i\right)=\phi_i^{\top} \sigma_i \sigma_i^{\top} \phi_i-2\left(\sigma_i \rho_i\right)^{\top} \phi_i$, $i \in \mathcal{M}$, and $C(\phi)=\operatorname{diag}\left(c_0\left(\phi_0\right), c_1\left(\phi_1\right)\right)$, $\phi=\left(\phi_0, \phi_1\right)$. Then the pair $\left(\phi^*, w^*\right)$ can be computed by solving
    $$
    \max _{w \in \mathbb{R}^{\prime}} \min _{\phi_0, \phi_1 \in \mathbb{R}_{+}^3} J(\boldsymbol{\phi}, w):=\left[\nu_0^{\top} e^{[Q+C(\boldsymbol{\phi})] T} \mathbf{1}\right]\left(x_0-w\right)^2-(w-z)^2,
    $$
    where $\nu_0$ denotes the initial distribution over market regimes. In our numerical implementation, the initial regime is sampled uniformly from the two market states, so that
    $\nu_0=\left(\frac{1}{2}, \frac{1}{2}\right)^{\top}$. The inner minimization is separable across regimes. Hence, for each $i \in \mathcal{M}$,
    $$
    \phi_i^*=\arg \min _{\phi_i \in \mathbb{R}_{+}^3}\left\{\phi_i^{\top} \sigma_i \sigma_i^{\top} \phi_i-2\left(\sigma_i \rho_i\right)^{\top} \phi_i\right\} .
    $$
    
    For parameter settings and initialization, the target expected terminal wealth is set to $z \in \{1.05, 1.25, 1.40, 2.00\}$ to comprehensively evaluate the algorithm's performance across diverse return preferences, including extreme return targets. The initial value of the Lagrange multiplier is set to $w=z+0.07$. 
    The policy parameters $\phi$ of the Actor network are uniformly and randomly initialized in the interval $[0.01, 0.05]$, whereas the value evaluation parameters $\theta$ of the Critic network are initialized to zero. 
    The whole training process encompasses a warm-up phase and a multi-timescale alternating optimization phase. The number of warm-up iterations for the Critic network is set to 50. In the subsequent main loop, the outer Actor optimization is run for 4500 episodes. Within each episode, the Critic network undergoes 5 inner iterations to ensure precise value evaluation. The macroscopic multiplier $w$ is updated every 15 episodes using the stochastic approximation (Robbins-Monro) method based on the clipped empirical mean of the terminal wealth. The learning rates for the Critic, Actor, and multiplier are set to $\eta_\theta=0.005$, $\eta_\phi=0.5$ (with a linear decay strategy to stabilize convergence in later stages), and $\eta_w=0.05$, respectively.

    Results for different expected returns. Figures \ref{fig:target_return_z1.05}-\ref{fig:target_return_z2.00} show the iterates of $\phi, w, \mathbb{E}\left[x_T\right]$, and $\operatorname{Var}\left(x_T\right)$ for $z=1.05, 1.25, 1.40$, and $2.00$, respectively. Table \ref{tab:final_estimates_transposed} reports the final estimates. These results demonstrate that our algorithm converges to the ground-truth values without knowledge of the model parameters. In particular, it learns the optimal values of the terminal mean and variance very well.

    \begin{table}[htbp]
    \caption{Final estimates of regime-switching parameters for different expected returns.}
    \label{tab:final_estimates_transposed}
    \centering
    \resizebox{\textwidth}{!}{ 

    \begin{tabular}{lcccc} 
        \toprule
        & \multicolumn{4}{c}{Expected Return ($z$)} \\
        \cmidrule(lr){2-5} 
        Parameter & $1.05$ & $1.25$ & $1.40$ & $2.00$ \\
        
        \midrule
        \addlinespace 

        $\phi^*$ (Regime 0) & (0.000, 1.341, 0.827) & (0.000, 1.341, 0.827) & (0.000, 1.341, 0.827) & (0.000, 1.341, 0.827) \\
        [0.5ex]
        $\phi$ (Regime 0) & (0.134, 1.237, 0.736) & (0.067, 1.189, 0.753) & (0.098, 1.220, 0.698) & (0.135, 1.347, 0.624) \\
        $\phi^*$ (Regime 1) & (0.000, 1.341, 0.827) & (0.000, 1.341, 0.827) & (0.000, 1.341, 0.827) & (0.000, 1.341, 0.827) \\ 
        $\phi$ (Regime 1) & (0.272, 1.340, 0.767) & (0.333, 1.390, 0.682) & (0.350, 1.363, 0.693) & (0.366, 1.252, 0.833) \\
        
        \addlinespace
        
        $w^*$ & 1.197 & 1.984 & 2.574 & 4.934 \\
        $w$ & 1.195 & 1.976 & 2.570 & 4.831 \\
        $\mathbb{E}[x_T^*]$ & 1.050 & 1.250 & 1.400 & 2.000 \\
        $\mathbb{E}[x_T]$ & 1.047 & 1.238 & 1.381 & 1.964 \\
        $Var[x_T^*]$ & 0.007 & 0.183 & 0.469 & 2.934 \\
        $Var[x_T]$ & 0.010 & 0.242 & 0.607 & 3.819 \\
        
        \bottomrule
     \end{tabular}
     } 
     \end{table}

    \begin{figure}[htbp]
        \centering
        \includegraphics[width=\textwidth, height=0.92\textheight, keepaspectratio]{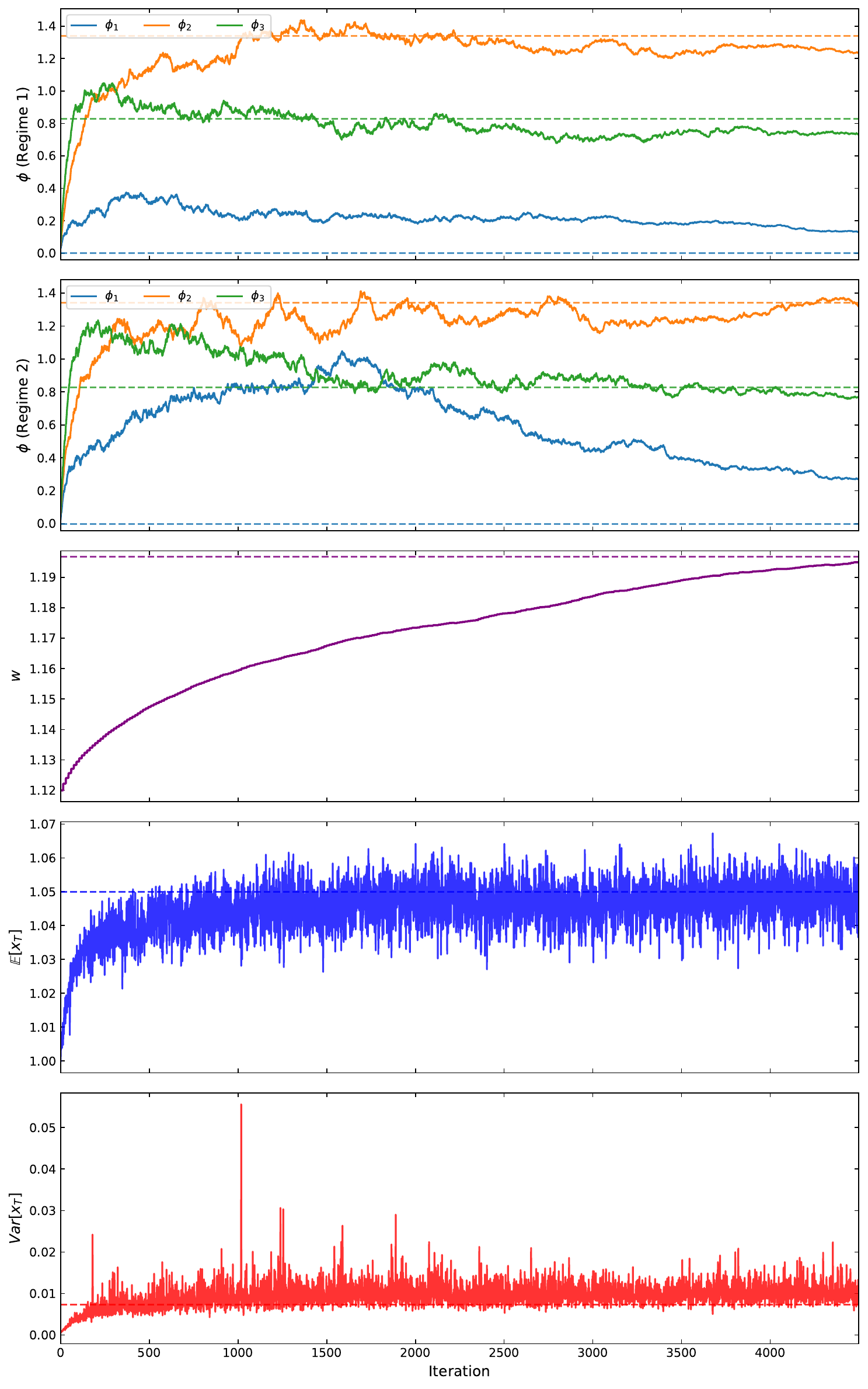}
        \caption{Iterations of $\phi, w, \mathbb{E}[x_T]$, and $\operatorname{Var}[x_T]$ in Algorithm \ref{algo_ba2c_modified} for the target expected return $z=1.05$. The first two panels show the three components of $\phi$ in regime 0 and regime 1, respectively. Horizontal dashed lines indicate the corresponding reference values.}
        \label{fig:target_return_z1.05}
    \end{figure}

    \begin{figure}[htbp]
        \centering
        \includegraphics[width=\textwidth, height=0.92\textheight, keepaspectratio]{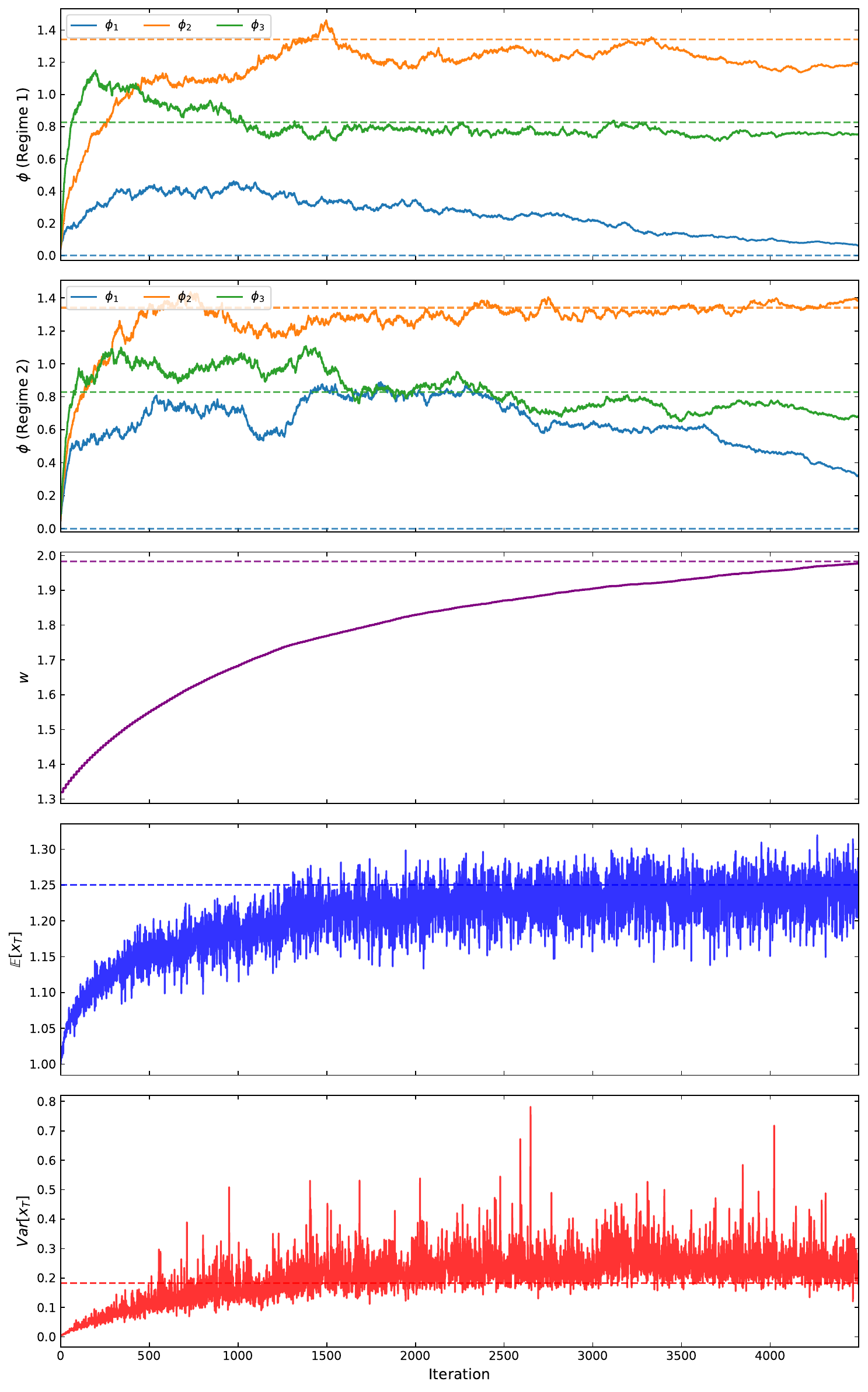}
        \caption{Iterations of $\phi, w, \mathbb{E}[x_T]$, and $\operatorname{Var}[x_T]$ in Algorithm \ref{algo_ba2c_modified} for the target expected return $z=1.25$. The first two panels show the three components of $\phi$ in regime 0 and regime 1, respectively. Horizontal dashed lines indicate the corresponding reference values.}
        \label{fig:target_return_z1.25}
    \end{figure}

    \begin{figure}[htbp]
        \centering
        \includegraphics[width=\textwidth, height=0.92\textheight, keepaspectratio]{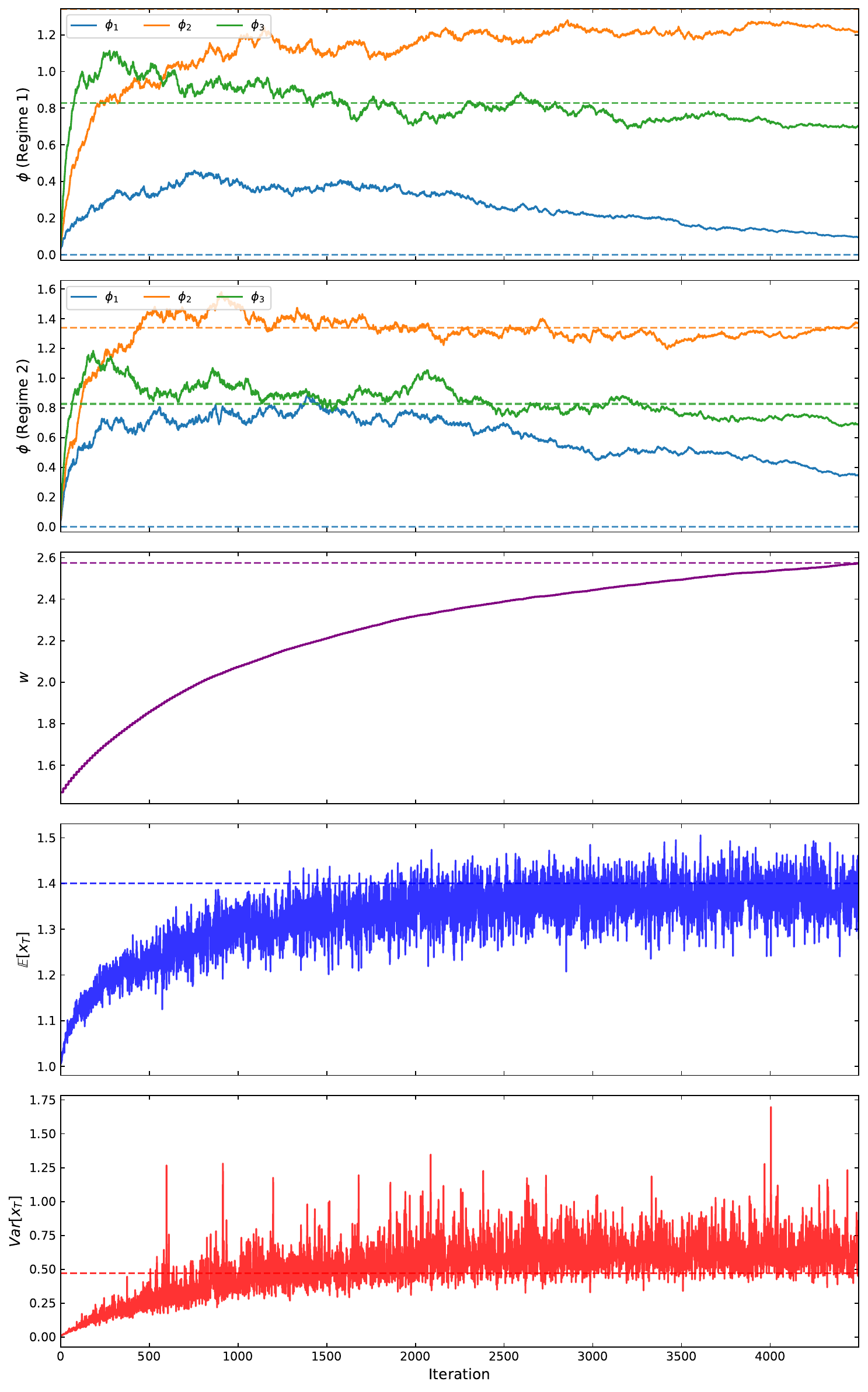}
        \caption{Iterations of $\phi, w, \mathbb{E}[x_T]$, and $\operatorname{Var}[x_T]$ in Algorithm \ref{algo_ba2c_modified} for the target expected return $z=1.40$. The first two panels show the three components of $\phi$ in regime 0 and regime 1, respectively. Horizontal dashed lines indicate the corresponding reference values.}
        \label{fig:target_return_z1.40}
    \end{figure}

    \begin{figure}[htbp]
        \centering
         \includegraphics[width=\textwidth, height=0.92\textheight, keepaspectratio]{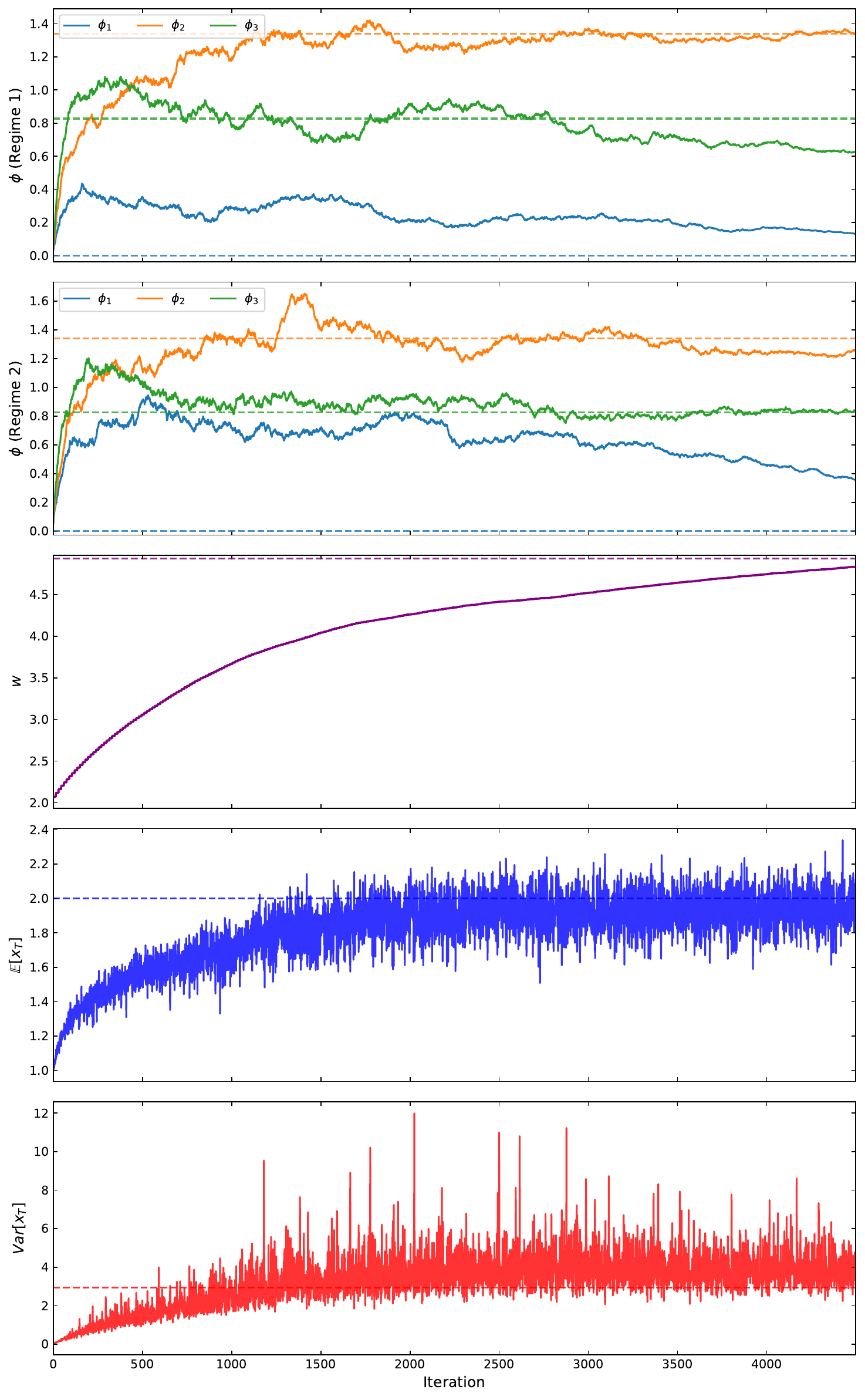}
        \caption{Iterations of $\phi, w, \mathbb{E}[x_T]$, and $\operatorname{Var}[x_T]$ in Algorithm \ref{algo_ba2c_modified} for the target expected return $z=2.00$. The first two panels show the three components of $\phi$ in regime 0 and regime 1, respectively. Horizontal dashed lines indicate the corresponding reference values.}
        \label{fig:target_return_z2.00}
    \end{figure}

\section{Conclusions.}\label{sec6}

    This paper studies the continuous-time mean–variance portfolio selection problem with no-shorting constraints in a Markov regime-switching market from a reinforcement learning perspective. The main result is that strict portfolio constraints need not be enforced on the support of the exploratory policy in order to recover the optimal constrained control. By imposing the no-shorting constraint only on the mean of the Gaussian policy, we retain the analytical tractability of Gaussian exploration while preserving that the mean of the optimal policy ultimately implemented in practice satisfies the no-shorting requirement. More importantly, with an appropriately chosen exploration variance, the mean of the optimal policy of the auxiliary problem coincides exactly with the optimal control of the target problem. This exact recovery result provides the theoretical link between tractable Gaussian exploration and the underlying constrained control problem. 

    Based on this theoretical basis, we develop a model-free continuous-time reinforcement learning algorithm. By integrating martingale-based policy evaluation, policy gradient optimization, and stochastic approximation for the Lagrange multiplier, our method avoids directly solving the piecewise coupled HJB system during learning. 
    The numerical results support the theoretical analysis and show that the proposed method can effectively recover the constrained optimal portfolio in the regime-switching market.

	Several promising directions for future research remain open. One natural extension is to move beyond the no-shorting constraint to more general closed convex cone constraints and portfolio frictions such as transaction costs. A second direction is to establish global convergence and finite-sample guarantees for the proposed multi-timescale Actor–Critic scheme in continuous-time constrained settings. Finally, extending the proposed auxiliary exploratory approach from the current mean–variance framework to more general constrained stochastic LQ control problems may provide a useful theoretical foundation for the development of continuous-time reinforcement learning methods with control constraints.

	\bigskip

\newpage
\bibliography{references}

\end{document}